\documentclass[12pt]{amsart}
\usepackage[margin=2.5cm]{geometry}
\usepackage[alphabetic]{amsrefs}
\usepackage{amssymb,mathrsfs,amsmath,amscd,multirow,amsthm}
\usepackage[dvipsnames]{xcolor}
\usepackage{graphicx}
\usepackage[all]{xy}
\usepackage{bigdelim}
\usepackage{verbatim} 
\usepackage{microtype} 
\usepackage[inline,shortlabels]{enumitem}%
\setlist{topsep=3pt, itemsep=0pt,partopsep=0pt}
\usepackage[colorlinks]{hyperref}
\hypersetup{linkcolor=BrickRed,citecolor=Green,filecolor=Mulberry,urlcolor=NavyBlue,menucolor=BrickRed,runcolor=Mulberry}

\newtheorem{thm}{Theorem}[section]

\newtheorem{prop}[thm]{Proposition}
\newtheorem*{ques*}{Question}

\newtheorem{lem}[thm]{Lemma}
\newtheorem{coveringThm}[thm]{Covering Theorem}

\theoremstyle{definition}

\newtheorem{examp}[thm]{Example}

\theoremstyle{remark}
\newtheorem{rmk}[thm]{Remark}

\newtheorem*{rmk*}{Remark}

\theoremstyle{plain}

\DeclareMathOperator\Spec{Spec}

\DeclareMathOperator\rank{rank}
\DeclareMathOperator\Char{char}

\DeclareMathOperator\Ann {Ann}
\DeclareMathOperator\Supp{Supp}
\DeclareMathOperator\im  {Im}

\let\geL\succeq \let\leL\preceq
\def\geQ{\geL_{\mathbb Q}}

\def\simQ{\sim_{\mathbb Q}}

\def\abs#1{\left\lvert#1\right\rvert}

\newcommand\oneoverp[1][]{{\mathchoice{\frac1{p#1}}{\smash{\frac1{p#1}}\vphantom{\hbox{f}}}{\smash{\frac1{p#1}}\vphantom{\raise-2pt\hbox{p}}}{\smash{\frac1{p#1}}\vphantom{\raise-1pt\hbox{p}}}}}%
\title[On canonical bundle formula]{%
	On canonical bundle formula for fibrations of curves with arithmetic genus one}

\author{Jingshan Chen}
\email{chjingsh@hbmzu.edu.cn}
\address{School of Mathematics and Statistics, Hubei Minzu University, Enshi, Hubei Province, China}

\author{Chongning Wang}
\email{chongningwang@hbmzu.edu.cn}
\address{School of Mathematics and Statistics, Hubei Minzu University, Enshi, Hubei Province, China}

\author{Lei Zhang}
\email{zhlei18@ustc.edu.cn}
\address{School of Mathematical Science, University of Science and Technology of China, Hefei 230026, P.R.China.}


\begin{document}
		
		\begin{abstract}
		In this paper, we develop canonical bundle formulas for fibrations of relative dimension one in characteristic $p>0$. For such a fibration from a log pair $f\colon (X, \Delta) \to S$, if $f$ is separable, we can obtain a formula similar to the one due to Witaszek \cite{Wit21}; if $f$ is inseparable, we treat the case when $S$ is of maximal Albanese dimension. As an application, we prove that for a klt pair $(X,\Delta)$ with $-(K_X+\Delta)$ nef, if the Albanese morphism $a_X\colon X \to A$ is of relative dimension one, then $X$ is a fiber space over $A$.
		\end{abstract}
		\maketitle
		
		\section{Introduction}\label{sec:intro}
		The canonical bundle formula over the field $\mathbb{C}$ of complex numbers is developed to study fibrations whose general fibers have numerically trivial (log-\nobreak)canonical divisors.
		Roughly speaking, for a fibration $f\colon X\to S$ of projective varieties with certain mild singularities such that $K_X \sim_{\mathbb Q} f^*D$ for some divisor $D$ on $S$, the canonical bundle formula predicts that $D \sim_{\mathbb Q} K_S + \Delta_S$, where $\Delta_S$ contains the information of singular fibers and moduli of general fibers, and $(S, \Delta_S)$ is expected to have mild singularities (\cite{FM00, Amb04, Amb05}).
		The canonical bundle formula plays an important role in birational geometry, for example, in the proof of subadjunction and effectivity of pluricanonical systems (\cite{Kaw98, BZ16}).
		We refer the interested reader to \cite{FL20} for a nice survey.
		Remark that over $\mathbb{C}$, to derive the information about $\Delta_S$, the key ingredients are results from moduli theory and Hodge theory (variations of Hodge structure).
		
		In characteristic $p>0$, there are quite few results about canonical bundle formulas.
		For elliptic fibrations, by taking advantage of the moduli theory of elliptic curves, Chen-Zhang \cite{CZ15} proved that $\Delta_S$ is $\mathbb Q$-linearly equivalent to an effective divisor (denoted by $\Delta_S \geQ 0$ for short).
		Cascini-Tanaka-Xu \cite[Section~6.2]{CTX15} investigated fibrations of log canonical pairs $f\colon (X, \Delta) \to S$ of relative dimension one.  Assuming that the geometric generic fiber $(X_{\overline{K(S)}}\cong \mathbb{P}^1_{\overline{K(S)}}, \Delta_{\overline{K(S)}})$ has log canonical singularities, they proved similar results as in characteristic zero by use of the moduli theory of stable rational curves. However, the geometric generic fiber can be quite singular.
		For example, in characteristic $p<5$, there exist quasi-elliptic fibrations (\cite[Chapter~7]{Ba01}), whose general fibers are singular rational curves with arithmetic genus one.
		For a quasi-elliptic fibration $f\colon X\to S$, the relative canonical divisor $K_{X/S}$ is not necessarily $\mathbb Q$-linearly equivalent to an effective $\mathbb Q$-divisor.
		To treat fibrations fibered by wildly singular varieties (log pairs),
		Witaszek \cite{Wit21} obtained the following result.
		\begin{thm}[{\cite[Theorem~3.4]{Wit21}}]
			\label{thm:Wit}
			Let $(X, \Delta)$ be a projective log canonical pair defined over an algebraically closed field $k$ of characteristic $p > 0$,
			and let $f\colon X \to S$ be a fibration of relative dimension one such that $K_X + \Delta \sim_{\mathbb{Q}} f^*D$ for some $\mathbb{Q}$-Cartier $\mathbb{Q}$-divisor $D$ on $S$.
			Assume that the generic fiber is smooth.
			Then there exists a finite purely inseparable morphism $\tau \colon T \to S$ such that
			$$\tau^*D\sim_{\mathbb{Q}} t\tau^*K_S + (1-t)(K_T+\Delta_T)$$
			for some rational number $t\in [0,1]$ and an effective $\mathbb{Q}$-divisor $\Delta_T$ on $T$.
		\end{thm}
		Witaszek noticed that if $(X_{\overline{K(S)}}, \Delta_{\overline{K(S)}})$ is not log canonical, then $\Delta$ has a horizontal component $T$ purely inseparable over $S$. By doing the base change $T\to S$ and applying the adjunction formula, he proved the above canonical bundle formula.
		Stimulated by Witaszek's observation, we attempt to treat fibrations of relative dimension one with singular general fibers, which happen only when $p=2$ or $3$. Here we remark that if $\dim S >1$, then the fibration $f$ is not necessarily a separable morphism; this means that the generic geometric fiber $X_{\overline{K(S)}}$ is not necessarily reduced. To treat inseparable fibrations, we apply the language of foliations, developed in \cite{PW22, JW21}, which makes it convenient to compare the canonical divisors under purely inseparable base changes.
		\smallskip
		
		First, we get a similar result for separable fibrations as Theorem~\ref{thm:Wit}.
		\begin{thm}[see Theorem~\ref{thm:sep-cb-formula}]\label{thm:intro:sep-cb-formula}
			Let $f\colon X\to S$ be a separable fibration of relative dimension one between normal quasi-projective varieties where $X$ is $\mathbb Q$-factorial.
			Let $\Delta$ be an effective $\mathbb{Q}$-divisor on $X$.
			Assume that
			\begin{enumerate}[\rm(C1)]
				\item $(X_{K(S)}, \Delta_{K(S)})$ is lc; and
				\item $K_X+\Delta\sim_{\mathbb{Q}} f^*D$ for some $\mathbb{Q}$-Cartier $\mathbb{Q}$-divisor $D$ on $S$.
			\end{enumerate}
			Then there exist finite purely inseparable morphisms $\tau_1\colon \bar{T} \to S$, $\tau_2\colon \bar{T}' \to \bar{T}$, an effective $\mathbb{Q}$-divisor $E_{\bar T'}$ on $\bar T'$ and rational numbers $a,b,c\geq 0$ such that
			$$\tau_2^* \tau_1^*D \sim_{\mathbb{Q}} a K_{\bar{T}'} + b\tau_2^*K_{\bar{T}} + c\tau_2^*\tau_1^*K_S + E_{\bar T'}.$$
			Moreover, if $(X_{K(S)}, \Delta_{K(S)})$ is klt, then $c \geq c_0$ for some positive number $c_0$ relying only on the maximal coefficient of prime divisors in $\Delta_{K(S)}$.
		\end{thm}
		
		For inseparable fibrations, we can only treat fibrations with the base $S$ being of maximal Albanese dimension (m.A.d.).
		We distinguish the cases according to whether the Albanese morphism $a_S$ of $S$ is separable or not and summarize the results in the following theorem.
		\begin{thm}[see Subsection~\ref{sec:proof-cbf}]\label{thm:main-thm-cbf}
			Let $X$ be a normal $\mathbb{Q}$-factorial projective variety and $\Delta$ an effective $\mathbb{Q}$-divisor on $X$.
			Let $f\colon X \to S$ be an inseparable fibration of relative dimension one onto a normal projective variety $S$.
			Assume that
			\begin{enumerate}[\rm(C1)]
				\item $(X_{K(S)}, \Delta_{K(S)})$ is lc;
				\item there exists a big open subset $S^{\circ}$ and a $\mathbb{Q}$-Cartier $\mathbb{Q}$-divisor $D^{\circ}$ on $S^{\circ}$ such that
				$(K_X+\Delta)|_{f^{-1}(S^{\circ})} \sim_\mathbb Q f^*D^{\circ}$;
				\item $S$ is of maximal Albanese dimension.
			\end{enumerate}
			Denote by $a_S\colon S \to A$ the Albanese morphism of $S$ and $D:=\overline{D^{\circ}}$ the closure divisor of $D^{\circ}$ in $S$. Then the following statements hold true:
			\begin{enumerate}[\rm(1), itemsep=2pt]
				\item If $a_S$ is separable, then $D - \frac{1}{2p}K_S \geQ 0$ {\rm(}$\mathbb Q$-linearly to an effective divisor{\rm)}. In particular, $\kappa(S,D) \geq \kappa(S)$.
				\item If $a_S$ is inseparable then $\kappa(S,D) \geq 0$, and the equality is attained only when $a_S: S \to A$ is purely inseparable of height one; and if moreover $(X_{K(S)}, \Delta_{K(S)})$ is klt, $a_S$ is finite and $K_X+\Delta \sim_\mathbb Q f^*D$, under the assumption that resolutions of singularities hold in $\dim S$, we have $\kappa(S,D)\ge1$.
			\end{enumerate}
		\end{thm}
		\begin{rmk}
			When $a_S$ is inseparable and $\kappa(S,D)= 0$, we can derive additional information, see Theorem~\ref{thm:insep-A}.
		\end{rmk}
		
		Varieties with a nef anti-canonical divisor are of special interest and expected to have good structures. For example, over $\mathbb{C}$, for a projective klt pair $(X,\Delta)$, if $-(K_X + \Delta)$ is nef, then the Albanese morphism $a_X\colon X\to A$ is a fibration which has certain isotrivial structure (\cite{Amb05, Cao19, CCM21}). In characteristic $p>0$, under the condition that the geometric generic fiber has certain mild singularities, similar results hold (see \cite{Wang22, Ejiri19}). In a recent paper, under certain conditions on singularities, Ejiri and Patakfalvi (\cite{EP23}) prove that $a_X\colon X\to A$ is surjective, and if $X\buildrel f\over\to S\buildrel g\over\to A$ is the Stein factorization of $a_X$, then $g$ is purely inseparable. More precisely, \cite{EP23} shows that if $f\colon X\to S$ or $g\colon S\to A$ is separable, then $a_X\colon X\to A$ is a fibration. In this paper, by applying the canonical bundle formula established above and an additional careful analysis of certain purely inseparable base changes and the related foliations, we can prove that if $a_X\colon X\to A$ is of relative dimension one, then it is a fibration.
		
		\begin{thm}[see Theorem~\ref{thm:S-abelian}]\label{thm:intro:S-abelian}
			Let $(X, \Delta)$ be a projective normal $\mathbb Q$-factorial klt pair. Assume that $-(K_X+\Delta)$ is nef. If the Albanese morphism $a_X\colon X \to A$ is of relative dimension one over the image $a_X(X)$.
			Then $a_X\colon X\to A$ is a fibration.
		\end{thm}


		\smallskip
		This is an important preparation for further studies of varieties with nef anti-canonical divisors. In a later paper \cite{CWZ26}, we focus on the case $K_X\equiv0$ and show that $X$ can be obtained by a sequence of quotients of group actions and foliations beginning with a product of an abelian variety and an elliptic (or a rational) curve, and we give specific descriptions of those quotients in some special cases.
		
		\subsection*{Conventions}
		\begin{itemize}[itemsep=4pt]
			\item
			For a scheme $Z$, we use $Z_{\mathrm{red}}$ to denote the scheme with the reduced structure of $Z$.
			\item By a \emph{variety} over a field $k$, we mean an integral quasi-projective scheme over $k$.
			For a variety $X$, we use $K(X)$ to denote the function field of $X$, and for a morphism $f\colon X \to S$  of varieties, we use $X_{K(S)}$ to denote the generic fiber of $f$.
			\item By a \emph{fibration}, we mean a projective morphism $f\colon X \to S$ of normal varieties such that $f_*\mathcal O_X=\mathcal O_S$, which implies that $K(S)$ is algebraically closed in $K(X)$.
			\item
			A morphism $f\colon X\to Y$ of varieties is said to be {\it separable} (resp.\ {\it inseparable}) if the field extension $K(X)/K(f(X))$ is {\it separable} (resp.\ {\it inseparable}).
			\item
			Let $f\colon X\to S$ be a fibration. A divisor $D$ on $X$ is said to be {\em $f$-exceptional} (resp.\ {\em vertical, horizontal}\/) if $f(\Supp D)$ is of codimension $\ge2$ in $S$ (resp.\ of codimension $\ge1$ in $S$, dominant over $S$).
			\item
			Let $k$ be a field of characteristic $p>0$ and $X$ be a variety over $k$.  We denote by $F_X\colon X^{\oneoverp} \to X$ the absolute Frobenius morphism.
			\item
			Let \( f\colon X \to Y \) be a morphism. The pullback \( f^*D \) is well-defined under one of the following conditions:
			\begin{itemize}
				\item[(1)] \( D \) is a \( \mathbb{Q} \)-Cartier divisor on \( Y \), or
				\item[(2)] both \( X \) and \( Y \) are normal, \( D \) is a \( \mathbb{Q} \)-divisor, and \( f \) is equidimensional.
			\end{itemize}
			\item
			For a morphism $\sigma\colon Z \to X$ of varieties, if $D$ is a divisor on $X$ such that the pullback $\sigma^*D$ is well defined, we often use $D|_Z$ to denote $\sigma^*D$ for simplicity.
			\item
			Let $X$ be a normal variety and denote by $i\colon X^\circ\hookrightarrow X$ the inclusion of the regular locus of $X$.
			For a Weil divisor $D$ on $X$, $\mathcal{O}_X(D)$ is a subsheaf of the constant sheaf $K(X)$ of rational functions, with the stalk at a point $x$ being defined by \[
			\mathcal{O}_X(D)_x:=\left\{
			f \in K(X)\biggm|
			\vcenter{\hbox{%
					$(\mathrm{div}(f) + D)|_U \geq 0$ for some}\hbox{open set $U$ containing $x$}
			}
			\right\}.
			\]
			We may identify $\mathcal{O}_X(D)=i_*\mathcal{O}_{X^\circ}(D|_{X^\circ})$.
			\item
			We denote by $\sim$, $\sim_{\mathbb{Q}}$ and $\equiv$ the linear, $\mathbb{Q}$-linear and numerical equivalence of divisors, respectively.
			\item
			Let $A = \sum a_i C_i$ and $B = \sum b_i C_i$ be effective divisors. We define $A \land B := \sum \min (a_i, b_i) C_i$.
			\item
			For two $\mathbb{Q}$-divisors $D,D'$ on a normal variety $X$, by $D \geq D'$ we mean that $D-D'$ is an effective divisor;
			and by $D\succeq D'$ (resp.\ $D\succeq_{\mathbb{Q}} D'$) we mean that $D-D'$ is linearly (resp.\ $\mathbb{Q}$-linearly) equivalent to an effective divisor.
			\item
			Let $X$ be a normal variety.
			If a coherent sheaf $\mathcal F$ on $X$ is reflexive (of rank $r$), then we denote $\det \mathcal F := (\bigwedge^r \mathcal F)^{\vee\vee}$.
			When $\mathcal{F}$ is just locally free in codimension one (e.g., torsion free),
			we define $\det \mathcal{F} := i_*\det(\mathcal{F}|_U)$, where $i\colon U\hookrightarrow X$ is an inclusion of a {\it big open subset} (meaning its complement in $X$ has codimension at least $2$) on which $\mathcal F$ is locally free.
			\item
			Let $X$ be a normal projective variety.
			The Kodaira dimension $\kappa(X)$ of $X$ is defined to be the Iitaka dimension $\kappa(X,K_X)$.
			More precisely, it is the dimension of the image of $\Phi\colon X\dashrightarrow \mathbb P H^0(X, nK_{X})$ for $n$ big and divisible enough. 
			\item 
			Let $X$ be a normal projective variety and let $D$ be an $\mathbb R$-Cartier $\mathbb R$-divisor on $X$.
			Let $H$ be an ample Cartier divisor on $X$.
			If $D$ is nef, then the {\it numerical dimension} of $D$ is the largest integral number $j\ge0$ such that $(D^j\cdot H^{n-j}) \ne 0$, see \cite[Remark~4.6]{CHMS14}.
			\item 
			Let $X$ be a normal projective variety over an algebraically closed field $k$.
			We say that $X$ is of maximal Albanese dimension, abbreviated as m.A.d., if the Albanese morphism $a_X\colon X\to A_X$ of $X$ is generically finite, or equivalently if $\dim a_X(X) = \dim X$.
		\end{itemize}
		\section{Preliminaries}\label{sec:pre}
		In this section, we collect some basic results about divisors and linear systems that will be used in the sequel. We work over an algebraically closed field $k$.
		
		\begin{lem}[{\cite[Lemma~4.2]{Zha19}}]\label{lem:kod-gg}
			Let $\mathcal E$ be a coherent sheaf which is locally free in codimension one on a normal projective variety $X$.
			Assume that $\mathcal E^{\vee\vee}$ is generically globally generated and $h^0(X, \mathcal E^{\vee\vee}) > \rank \mathcal E$. Then $h^0(X, \det\mathcal E) >1$.
		\end{lem}
		
		\begin{lem}\label{lem:push-down}
			Let $\sigma\colon Y \to X$ be a proper dominant morphism of normal varieties, generically finite of degree $d$.
			Let $D, D'$ be $\mathbb{Q}$-Cartier $\mathbb{Q}$-divisors on $X, Y$ respectively.
			Assume that there exists a divisor $N$ on $Y$ exceptional over $X$ such that $\sigma^*D \sim_{\mathbb Q} D' + N$.
			Then $\sigma_*D' \sim_{\mathbb Q} dD$.
		\end{lem}
		\begin{proof}
			See \cite[Theorem~1.4]{Fulton98}.
		\end{proof}
		
		\begin{coveringThm}[{\cite[Theorem~10.5]{Iit82}}]\label{thm:covering}
			Let $f\colon Y \to X$ be a proper surjective morphism between normal complete varieties.
			If $D$ is a Cartier divisor on $X$ and $E$ an effective $f$-exceptional divisor on $Y$,
			then
			$$\kappa(Y,f^*D + E) = \kappa(X,D).$$
		\end{coveringThm}
		
		\begin{rmk*}
			If furthermore, $f$ is equidimensional, so that pulling back Weil $\mathbb Q$-divisors makes sense,
			then the equality $\kappa(Y,f^*D+E) = \kappa(X,D)$ still holds for Weil $\mathbb Q$-divisors $D$.
			Indeed, by the proof of \cite[Theorem~10.5]{Iit82}, we have $\kappa(f^{-1}(X^{\rm reg}),f^*(D|_{X^{\rm reg}})+E) = \kappa(X^{\rm reg},D|_{X^{\rm reg}})$,
			where $X^{\rm reg}$ denotes the regular locus of $X$.
			Here, $\kappa(X^{\rm reg},D|_{X^{\rm reg}})$ etc., are well defined since $H^0(X,mD)\cong H^0(X^{\rm reg},mD|_{X^{\rm reg}})$ for any positive integer $m$.
		\end{rmk*}

		\begin{lem}\label{lem:conn}
			If a linear system $\mathfrak M$ (without fixed components) on a normal proper variety $X$ has a reduced and connected
			member $M_0$, then every $M\in\mathfrak M$ is connected.
		\end{lem}
		\begin{proof}
			To prove the connectedness of $M$, we consider the pencil $\Phi\colon X\dashrightarrow\mathbb P^1$ induced by $M$ and $M_0$.
			Let $\Gamma\subset X\times\mathbb P^1$ be the closure of the graph of $\Phi$ with projection $\widetilde\Phi\colon \Gamma \to \mathbb{P}^1$.
			Denote by $\widetilde{M_0}\subset\Gamma$ the strict transform of $M_0$.
			Then $\widetilde\Phi$ has a fiber $\widetilde{M_0}+E$ with $E$ exceptional over $X$.
			Since the fiber $\widetilde{M_0} + E$ is connected and non-multiple, by Stein factorization, each fiber of $\widetilde\Phi$ is connected.
			Therefore, $M$ is connected.
		\end{proof}
		
		While the subsequent two lemmas are standard results in the literature, we provide detailed proofs here for the reader's convenience and to make our exposition self-contained.
		\begin{lem}\label{lem:nefbigdiv}
			Let $X$ be a normal projective variety.
			If $D$ is a nef and big Cartier divisor on $X$, then for any $\epsilon>0$, there exists an effective $\mathbb{Q}$-Cartier $\mathbb{Q}$-divisor $D_\epsilon$ with coefficients $<\epsilon$ such that $D_\epsilon\sim_{\mathbb{Q}} D$.
		\end{lem}
		\begin{proof}
			Since $D$ is nef and big, there exists an effective $\mathbb{Q}$-Cartier divisor $N$ such that $D-N$ is ample.
			Therefore, the $\mathbb Q$-divisor $A_k:=D - \frac{1}{k}N=\frac{k-1}{k}D+\frac{1}{k}(D-N)$ is ample for $k\geq 1$.
			As $D=A_k+ \frac{1}{k}N$, it is easy to find the desired divisor $D_\epsilon$.
		\end{proof}

		\begin{lem}\label{lem:relativetrivial}%
			Let $f\colon X\to S$ be a fibration of normal projective varieties and $L$  a nef $\mathbb Q$-Cartier $\mathbb{Q}$-divisor on $X$.
			Assume that $L|_{X_{K(S)}}\simQ  0$, where $X_{K(S)}$ is the generic fiber.
			Then there exists a big open subset $S^{\circ} \subset S^{\rm reg}$ and a pseudo-effective divisor $D$ on $S$ such that
			$L|_{f^{-1}(S^{\circ})}\simQ  f^*D|_{S^{\circ}}$.
		\end{lem}
		\begin{proof}
			Applying the flattening trick of \cite{RG71} (cf.\ \cite[Theorem~2.3]{Wit21}), there exists a commutative diagram
			\[
			\xymatrix{
				X_2\ar[r]^{\smash{h'_2}}\ar[d]_{f_2}&X_1\ar[r]^{\smash{h'_1}}\ar[d]_{f_1} &X\ar[d]_{f}\\
				S_2\ar[r]^{h_2}&S_1\ar[r]^{h_1} &S\rlap,
			}\]
			where $h_1$ is a projective birational morphism, $f_1$ is flat, $X_1$ is the closure of the generic fiber $X_{K(S)}$ of $f$ in $X\times_S S_1$, $h_2$ is the normalization morphism and $X_2$ is the normalization of $X_1\times_{S_1}S_2$.
			Let $h = h_1\circ h_2$ and $h' = h'_1\circ h'_2$.
			Now $f_2$ is equidimensional, therefore by \cite[Lemma~2.18]{Wit21}, there exists a $\mathbb Q$-divisor $D_2$ on $S_2$ such that $h'^* L \simQ f_2^* D_2$.
			Since $f_2^*D_2$ is $\mathbb Q$-Cartier and $f_2$ is equidimensional, $D_2$ is $\mathbb Q$-Cartier too (see \cite[Lemma~2.6]{Druel21}).
			Then, since $f_2^*D_2$ is nef, by \cite[\S~I.4, Proposition~1]{Kleiman66}, $D_2$ is nef, and it follows that $D := h_*D_2$ is pseudo-effective.
			Take $S^\circ \subseteq S^{\rm reg}$ to be a big open subset over which $h_1$ is an isomorphism, then $L|_{f^{-1}(S^{\circ})}\simQ  f^*(D|_{S^{\circ}})$.
		\end{proof}
		
		Recall the adjunction formula and inversion of adjunction as follows.
		\begin{prop}[{\cite[Proposition~4.5]{Kollar13} and \cite[Theorem~4.1]{Das15}}]\label{lem:adjunction}
			Let $X$ be a normal variety and $S$ be a prime Weil divisor on $X$.
			Let $S^\nu \to S$ be the normalization map. Assume that $K_X + S$ is $\mathbb{Q}$-Cartier.
			Then 
			\begin{itemize}
				\item[\rm(1)] There exists an effective $\mathbb Q$-divisor $\Delta_{S^\nu}$ on $S^\nu$ such that
				$$(K_X+ S)|_{S^{\nu}} \simQ K_{S^{\nu}} + \Delta_{S^\nu} ,$$
				where $\Delta_{S^\nu} = 0$ if and only if both $X,S$ are regular at codimension-one points of $S$.
				\item[\rm(2)] If moreover that the pair $(S^\nu, \Delta_{S^\nu})$ is strongly $F$-regular (e.g., $S^\nu$ is regular and $\Delta_{S^\nu}=0$), then $S$ is normal.
			\end{itemize}
		\end{prop}
		
		Applying the adjunction formula we obtain the following result, which will be used frequently to study the behavior, under purely inseparable morphisms, of the restriction of the (log-)canonical divisor on certain divisors.
		\begin{lem}\label{lem:log-adj}
			Let $X$ be a normal $\mathbb Q$-factorial quasi-projective variety, $\Delta$ an effective $\mathbb Q$-divisor on $X$ and $T$ a prime divisor on $X$.
			Write $\Delta= aT + \Delta'$ with $T\not\subset\Supp\Delta'$.
			Then there exists an effective divisor $B_{T^\nu}$ on the normalization $T^\nu$ of $T$ such that
			\begin{equation}\label{eq:TMIZ}
				(1-a)T|_{T^\nu} \sim_{\mathbb{Q}} K_{T^\nu} + B_{T^\nu} - (K_X+\Delta)|_{T^\nu}.
			\end{equation}
		\end{lem}
		\begin{proof}
			Applying the adjunction formula of Lemma~\ref{lem:adjunction}, we may write that
			$$\bigl( (K_X+ \Delta) + (1-a)T \bigr)|_{T^\nu}\sim_{\mathbb{Q}} (K_X + T + \Delta')|_{T^\nu} \sim_{\mathbb{Q}} K_{T^\nu} + \Delta_{T^\nu} +  \Delta'|_{T^\nu},$$
			where $\Delta_{T^\nu}\ge0$.
			We then deduce (\ref{eq:TMIZ}) by setting $B_{T^\nu} = \Delta_{T^\nu} +  \Delta'|_{T^\nu}$.
		\end{proof}
		
		The following results are consequences of \cite[Theorem~0.2]{HPZ19} and \cite[Proposition~3.2]{EP23}.
		\begin{prop}\label{prop:char-abel-var}
			Let $X$ be a normal projective variety with maximal Albanese dimension, namely the Albanese morphism $a_X\colon X\to A$ is generically finite.
			Then
			\begin{enumerate}[\rm(\arabic*)]
				\item the Kodaira dimension $\kappa(X, K_{X}) \geq 0$, and if $a_X$ is inseparable, then $\kappa(X,K_X) \ge 1$.
				\item Assume that $X$ admits a resolution of singularities $\rho\colon Y \to X$.
				If $\kappa(X,K_X) = 0$, then $X$ is birational to an abelian variety.
				\item If $K_X \equiv 0$, then $X$ is isomorphic to an abelian variety.
			\end{enumerate}
		\end{prop}
		\begin{proof}
			The statement (1) follows from \cite[Theorem~4.1]{Zha19}, where $X$ is assumed to be smooth, but the proof also applies to normal varieties.
			
			To show (2), let $\kappa(X,K_X) = 0$.
			We may chose $K_Y$ so that $\rho_*K_Y = K_X$, then $\kappa(X,K_X)\ge \kappa(Y,K_Y) \ge 0$.
			Thus $\kappa(Y,K_Y) = 0$.
			By \cite[Theorem~0.2]{HPZ19}, the Albanese morphism $a_Y\colon Y\to A_Y$ is birational.
			Note that, the composition morphism $Y\to X\to A$ factors into $Y\buildrel a_Y\over\to A_Y \buildrel\pi\over\to A$.
			We see that $\pi\colon A_Y \to A$ is surjective. 
			Moreover, since $a_X\colon X\to A$ is generically finite, $\pi$ is generically finite.
			Then, as $\pi$ is a morphism of abelian varieties, it is finite.
			It follows that $A_Y$ is the normalization of $A$ in $K(Y)=K(X)$.
			Thus, there is a birational morphism $X\to A_Y$.
			
			The statement (3) is \cite[Proposition~3.2]{EP23}.
		\end{proof}

		\section{The behavior of the canonical divisor under quotient of foliations and purely inseparable base changes}
		In this section, we investigate finite purely inseparable morphisms arising from base changes and compare the canonical divisors under these base changes.
		We borrow the notions and constructions from \cite{JW21}.
		As our setting mildly differs from that of \cite[Section~3.1]{JW21}, to avoid ambiguity, we sketch the construction of the divisors involved and include some statements in a reasonable order.
		
		Throughout this section, we work over a perfect field $k$ of characteristic $p>0$.
		
		\subsection{Foliations and purely inseparable morphisms}\label{sec:4Y4A}
		Let $Y$ be a normal variety over $k$.
		By a {\it foliation} on $Y$ we mean a saturated subsheaf $\mathcal F\subseteq \mathcal T_Y$ of the tangent bundle that is $p$-closed and involutive. Then $\Ann\mathcal F\subseteq \mathcal{O}_Y$ is a subsheaf of ring containing $\mathcal{O}_Y^p$.
		There is a one-to-one correspondence: \[
		\newcommand\mysatop[2]{\genfrac{}{}{0pt}{}{#1}{#2}}
		\left\{\mysatop{\text{foliations}}{\mathcal F\subseteq \mathcal T_Y}\right\}
		\leftrightarrow
		\left\{\vcenter{\hbox{finite purely inseparable morphisms $\pi\colon Y\to X$}%
			\hbox{over $k$ of height one with $X$ normal}}\right\},
		\]
		which is given by
		$$\mathcal F ~\mapsto ~\pi\colon Y \to \mathrm{Spec}( \Ann\mathcal F) \text{ \ and \ } \pi\colon Y\to X~\mapsto \mathcal F_{Y/X}:=\Omega_{X\to Y}^{\perp}$$
		where $\Omega_{X \to Y}:=\mathrm{im }(\pi^*\Omega_X^1 \to \Omega_Y^1)$ and $\Omega_{X\to Y}^{\perp}$ is the sheaf of tangent vectors in $\mathcal{T}_Y$ annihilated by $\Omega_{X \to Y}$.
		As a side note, $(\Omega_{Y/X}^1)^\vee \cong \mathcal F_{Y/X}$, and under the above correspondence $\rank \mathcal F_{Y/X} =\log_p\deg f$.
		Let $y\in Y$ be a smooth point and $x:=\pi(y)$.
		A foliation $\mathcal{F}\subseteq\mathcal T_Y$ is said to be \emph{smooth} at $y$ if around $y$ the subsheaf $\mathcal F$ is a subbundle, namely both $\mathcal{F}$ and $\mathcal{T}_{Y}/\mathcal{F}$ are locally free.
		It is known that (\cite[Page~142]{Ekedahl87}) \begin{equation}\label{eq:YHFP}
			\hbox{$\mathcal F$ is smooth at $y$ $\Leftrightarrow$ $Y/\mathcal F$ is smooth at $x$ $\Rightarrow$ $\Omega^1_{Y/X}$ is locally free at $y$} .
		\end{equation}

		Recall the following well-known result (cf. \cite[Proposition~2.10]{PW22}).
		\begin{prop}\label{prop:can-foliation}
			Let $\pi\colon Y\to X$ be a finite purely inseparable morphism of height one between normal varieties.
			Then
			\begin{equation}\label{eq:pullback-cano}
				\pi^*K_X \sim K_Y - (p-1)\det \mathcal F_{Y/X}\sim  K_Y + (p-1)\det \Omega_{Y/X}^1.
			\end{equation}
		\end{prop}
		
		\begin{rmk}\label{rmk:6.1}
			(1) To verify the linear equivalence of two Weil divisors on a normal variety, it suffices to do this in codimension one.
			So to treat $\det \mathcal F_{Y/X}$, we may assume that $\mathcal F_{Y/X}$ is smooth by working on a big open subset.
			
			(2) On a normal variety $X$, a foliation $\mathcal F$ is uniquely determined by its restriction $\mathcal F|_{X^\circ}$ over a dense open subset $X^\circ \subset X$ since $\mathcal F$ is saturated in $\mathcal T_X$.
			
			(3)
			Let $f\colon X\to S$ be a morphism of varieties.
			For any coherent sheaf $\mathcal G$ on $S$, there is a natural homomorphism
			\[
			\alpha : f^*(\mathcal G^\vee) \to (f^*\mathcal G)^\vee
			\]
			which is an isomorphism under one of the conditions: {\rm(i)} $f$ is flat; {\rm(ii)} $\mathcal G$ is locally free (see the proof of \cite[Proposition~1.8]{Hartshorne80}).
			
			Consequently, if either $f$ is flat or $\Omega^1_S$ is locally free,
			the dual of $df\colon f^*\Omega_{S}^1 \to \Omega^1_{X}$ gives a homomorphism \[
			\varphi : \mathcal T_X \to f^*\mathcal T_S.
			\]
		\end{rmk}

		\subsubsection{``Pushing-down'' foliations along a fibration}\label{sec:pushing-down}%
		Let $f\colon X\to S$ be a fibration of normal varieties and let $\mathcal F$ be a foliation on $X$.  We define a foliation $\mathcal G$ on $S$ as follows.
		
		The composition $X\to S \buildrel F_{S/k}\over\to S^{(1)}$ factors through the morphism $h\colon \bar{X}:= X/\mathcal{F} \to S^{(1)}$.
		Let $\bar{X} \buildrel\bar{f}\over\to \bar{S} \to S^{(1)}$ be the Stein factorization of $h$.
		Then there exists a foliation $\mathcal G$ on $S$ corresponding to the morphism $\sigma\colon S \to \bar{S}$.  In summary, we have
		\begin{equation}\label{eq:YL7C}
			\vcenter{\xymatrix{
					X \ar[r]^<<<<<\pi \ar[d]_f & \bar X = X/\mathcal F \ar[d]_{\bar f}\ar[rd]^h \\  S \ar[r]^<<<<<\sigma & \bar S = S/\mathcal G \ar[r] & S^{(1)}.
			}}
		\end{equation}
		
		We characterize $\mathcal G$ as follows.

		\begin{lem}\label{push-foliation}
			\def\subseteqgen{\mathrel{\subseteq_{\rm gen}}}
			We use the notation above. Assume further that $S$ is regular. 
			Let $\eta : \mathcal F \hookrightarrow \mathcal T_X \buildrel\varphi\over\to f^*\mathcal T_S$ be the composition homomorphism,
			where $\varphi$ is defined as in Remark~\ref{rmk:6.1}.
			Then
			\begin{enumerate}[\rm(1)]
				\item
				the sheaf $\mathcal G$ is the minimal foliation on $S$ 
				such that $\eta(\mathcal F) \subseteqgen f^*\mathcal H$ (here ``$\subseteqgen$'' means that the inclusion holds over certain open dense subset of $X$).
				\item
				if $f$ is separable, then $\bar{f}$ is separable if and only if $\eta\colon \mathcal F\to f^*\mathcal G$ is surjective generically.
			\end{enumerate}
		\end{lem}
		\begin{proof}
			Let $X^\circ \subset X^{\rm reg}$ be an open dense subset in the regular locus of $X$ where $f$ is flat.
			By \cite[Corollary~3.4]{Ekedahl87} we have the following commutative diagram with exact rows:
			\begin{equation}\label{eq:YPC0}
				\vcenter{\xymatrix@R4ex@C2ex{
						0\ar[r] & \mathcal F \cong \mathcal T_{X^\circ/\bar{X}^\circ} \ar[r]\ar[d]^{\varphi_0}\ar[rd]^{\eta} & \mathcal T_{X^\circ} \ar[r]\ar[d]^{\varphi_1} & \pi^*\mathcal T_{\bar{X}^\circ} \ar[r]\ar[d]^{\varphi_2} & F_X^* \mathcal T_{X^\circ/\bar{X}^\circ} \ar[r]\ar[d]^{\varphi_3} & 0\\
						0\ar[r] & f^*\mathcal G \cong f^*\mathcal T_{S/\bar S} \ar[r]& f^*\mathcal T_{S} \ar[r] & f^*\sigma^*\mathcal T_{\bar S} \cong \pi^*\bar f^*\mathcal T_{\bar S} \ar[r] & f^*F_S^*\mathcal T_{S/\bar S} \cong F_X^*f^*\mathcal T_{S/\bar S} \ar[r] & 0,
				}}
			\end{equation}
			where $F_X,F_S$ are the absolute Frobenius morphisms of $X,S$ respectively.
			
			\smallskip
			(1)
			\def\subseteqgen{\mathrel{\subseteq_{\rm gen}}}
			Note that for any two foliations $\mathcal H', \mathcal H''$ on $S$, the saturation of the sheaf $\mathcal H'\cap \mathcal H''$ is still a foliation.
			Let $\mathcal G'$ be the saturation of the intersection of those foliations $\mathcal H\subseteq \mathcal T_S$ such that
			$\eta(\mathcal F) \subseteqgen f^*\mathcal H$.
			Then $\mathcal G'$ is a foliation.
			We aim to show that $\mathcal G = \mathcal G'$.
			By the diagram (\ref{eq:YPC0}), we have $\eta(\mathcal F)\subseteqgen f^*\mathcal G$, thus $\mathcal G' \subseteq \mathcal G$.
			Let us show the inverse inclusion $\mathcal G \subseteq \mathcal G'$.
			Since $\eta(\mathcal F) \subseteqgen f^*\mathcal G'$, for any $s\in K(S)$ such that $\langle \mathcal G' ,s\rangle = 0$, we have $\langle \mathcal F,f^*(s)\rangle = 0$. This implies that $h\colon \bar X\to S^{(1)}$ factors through $\bar X \to S/\mathcal G'$. Since $\bar X \to S/\mathcal G$ is a fibration, we conclude that $\mathcal G \subseteq \mathcal G'$.
			
			\smallskip
			(2)
			Since $f$ is separable, $\varphi_1$ is generically surjective.
			Then $\varphi_0\colon \mathcal F\to f^*\mathcal G$ being generically surjective $\iff$ $\varphi_3$ being generically surjective $\buildrel\text{by ~}(\ref{eq:YPC0})\over\iff$ $\varphi_2$ being generically surjective  $\iff$ $\mathcal T_{\bar X} \to \bar f^*\mathcal T_{\bar S}$ being generically surjective  $\iff$ $\bar{f}$ being separable.
		\end{proof}

		\subsection{Foliations associated with morphisms arising from base changes}\label{sec:fol-bsch}
		An important kind of foliation is associated with morphisms arising from purely inseparable base changes.
		Let us briefly recall a relation of the canonical divisors built in \cite[Section~3 and Section~4]{JW21} and \cite[Section~3]{PW22}.
		
		\subsubsection{}\label{subsec:fol-bsch}
		Let $X, S, T$ be normal varieties over $k$, $f\colon X\to S$ a dominant morphism, and $\tau\colon T\to S$ a finite purely inseparable morphism of height one.
		Let $Y$ be the normalization of $(X_T)_{\rm red}$.
		Consider the following commutative diagram
		\[\xymatrix{&Y\ar[r]^<<<<<\nu \ar@/^8mm/[rrr]|{\,\pi\,} \ar[rrd]_g\ar[r] &(X_T)_{\rm red}\ar[r] &X_T\ar[r]\ar[d]^{f_T} &X\ar[d]^{f}\\
			&&  &T\ar[r]^{\tau} &S\rlap{.} \\
		}\]
		Then there exists a natural homomorphism \begin{equation}
			\delta\colon g^*\Omega_{T/S}^1\to \Omega_{Y/X}^1
		\end{equation} as the composition
		\begin{equation}\label{eq:surjectivity-of-relative-cotangent-sheaves}
			g^*\Omega_{T/S}^1 \buildrel\cong\over\to
			\Omega_{X_T/X}^1\otimes\mathcal O_Y \twoheadrightarrow \Omega_{(X_T)_{\mathrm{red}}/X}^1\otimes\mathcal O_Y \to \Omega_{Y/X}^1,
		\end{equation}
		where the first arrow is induced by the isomorphism $f_T^*\Omega_{T/S}^1 \cong \Omega_{X_T/X}^1$ (\cite[\href{https://stacks.math.columbia.edu/tag/01V0}{Tag 01V0}]{stacks25}).
		Note that $\delta$ is surjective except over the preimage of the non-normal locus of $(X_T)_{\rm red}$ (see \cite[Lemma~4.5]{JW21}).

		
		Since $S$ is normal and $\tau$ is a finite purely inseparable morphism of height one, we may take a big open subset $S^\circ\subset S$ such that $S^\circ$ and $T^\circ:=\tau^{-1}S^\circ$ are both regular. This ensures that $\Omega^1_{T^\circ/S^\circ}$ is locally free (see Section~\ref{sec:4Y4A}). By Remark~\ref{rmk:6.1} (3), there is an isomorphism $g^*\mathcal F_{T^\circ/S^\circ} \buildrel\sim\over\to (g^*\Omega^1_{T^\circ/S^\circ})^\vee$. This isomorphism together with the dual of $\delta$ yields a homomorphism
		$$\gamma\colon \mathcal F_{Y^\circ/X^\circ} \to g^*\mathcal F_{T^\circ/S^\circ},
		$$
		where $ Y^\circ := g^{-1}(T^\circ)$ and $X^\circ := f^{-1}(S^\circ)$. 
		Since $\delta$ is generically surjective, its dual homomorphism $\gamma$ is injective.
		We regard $\mathcal F_{Y^\circ/X^\circ}$ as a subsheaf of $g^*\mathcal F_{T^\circ/S^\circ}$ and  let $\widetilde{F_{Y^\circ/X^\circ}}$ be the saturation of 
		$\mathcal F_{Y^\circ/X^\circ}$ in $g^*\mathcal F_{T^\circ/S^\circ}$.

		We conclude that
		\begin{enumerate}[(i)]
		\item
		the inclusion $\mathcal F_{Y^\circ/X^\circ} \subseteq \widetilde{F_{Y^\circ/X^\circ}}$ induces an effective divisor $E$ on $Y^\circ$ such that $\det \mathcal F_{Y^\circ/X^\circ} = \det \widetilde{F_{Y^\circ/X^\circ}}(-E)$,
		where $\Supp E$ is contained in the preimage of the non-normal locus of $(X_{T^\circ})_{\mathrm{red}}$; and
		\item
		the generic fiber $X_{K(T)}$ of $f_T$ is  reduced if and only if $\rank\mathcal F_{Y/X} = \rank g^*\mathcal F_{T/S}$ by comparing the degree of the morphisms $\pi$ and $\tau$.
		\end{enumerate}

		\subsubsection{The movable part and fixed part}
		The following is a slight generalization of \cite[Theorem~1.1]{JW21}, which follows from almost the same argument. For the convenience of the reader, we sketch the proof.
		\begin{prop}\label{prop:JiWal-refined-version}
		Let $f\colon X\to S$ be a fibration and we use the notation in \S\ref{subsec:fol-bsch}.
		Let $\Gamma \subseteq H^0(T,\Omega_{T/S}^1)$ be a finite-dimensional $k$-vector subspace.
		Assume that there is an open subset $U\subseteq T$ such that $\Omega_{U/S}^1$ is locally free and globally generated by $\Gamma$.
		Set $r= \rank \Omega_{Y/X}^1$ and $\Gamma_Y= \mathrm{Im}\bigl(\bigwedge^r \Gamma \to H^0(Y,\det \Omega_{Y/X}^1)\bigr)$.
		Let $\mathfrak M +F \subseteq \lvert \det \Omega_{Y/X}^1 \rvert$ be the sub-linear system determined by $\Gamma_Y$ with the fixed part $F$ and the movable part $\mathfrak M$.
		Then
		\begin{enumerate}[\rm(1)]
			\item $\nu(F)|_{(X_U)_{\rm red}}$ is supported on the codimension one part of the union of the non-normal locus of $(X_U)_{\rm red}$ and the exceptional locus over $U$;
			\item the $g$-horizontal part $M_{h}$ of $M \in \mathfrak M$ is zero if and only if $X_{K(T)}$ is reduced.
		\end{enumerate}
		\end{prop}
		\begin{proof}
		Consider the following composition homomorphism
		\[
		\gamma : \bigwedge\nolimits^r \Gamma\otimes_k \mathcal{O}_{Y_U} \to
		g^*\bigwedge\nolimits^r\Omega_{U/S}^1 \to \Bigl(\bigwedge\nolimits^r\Omega_{Y/X}^1\Bigr)^{\vee\vee}\Big|_{Y_U}=\det \Omega_{Y/X}^1 \Big|_{Y_U}.
		\]
		Note that $\Supp F$ corresponds to the codimension one part of the locus where the above homomorphism is not surjective. More precisely, $\mathrm{Im}(\gamma) = \det \Omega_{Y_U/U}^1(-F)$ holds up to codimension one. Combining this with the result (i) in \S\ref{subsec:fol-bsch}, we conclude the assertion (1).
		
		Let us prove the assertion (2). If $X_{K(T)}$ is reduced, which is equivalent to that $\rank \Omega_{U/S}^1 = \rank \Omega_{Y/X}^1=r$,
		then $\mathrm{Im}(\gamma)_{K(T)}= K(T)(\alpha_1\wedge \cdots \wedge \alpha_r)$ for some $\alpha_1, \ldots, \alpha_r \in \Gamma$ generating $\Omega_{U/S}^1$ over the generic point of $U$, thus $M_{h}=0$.
		For the converse direction, assume that $X_{K(T)}$ is non-reduced. Then $m=\rank \Omega_{U/S}^1 >\rank \Omega_{Y/X}^1$. Remark that the argument of \cite[Section~5]{JW21} shows precisely the following statement
		\begin{itemize}
			\item
			if $e_1,\ldots, e_m$ are local basis of $\Omega_{U/S}^1$, then the sections like $\gamma(e_{i_1}\wedge \cdots \wedge e_{i_r})$ produce a nontrivial horizontal movable part of $\lvert \det\Omega_{Y_U/X}^1 \rvert$.
		\end{itemize}
		We conclude the proof after noticing that the sections like $\alpha_1\wedge \cdots \wedge \alpha_r$ for $\alpha_1, \ldots, \alpha_r$ in $\Gamma$ generate $\bigwedge^r\Omega_{U/S}^1$.
		\end{proof}

		\subsection{The behavior of the relative canonical divisors under base changes}
		The following result can be proved by use of duality theory similarly to \cite[Theorem~2.4]{CZ15}\footnote{The proof of \cite[Theorem~2.4]{CZ15} contains a minor mistake: the statement does not necessarily hold if the base is not smooth, which is mended in \cite[Proposition~2.4]{Zha19}.}.
		But here we give a proof by use of the language of foliation.
		
		\begin{prop}\label{prop:compds}
		Let $f\colon  X \rightarrow S$ be a fibration of normal varieties.
		Let $T$ be a normal variety and $\tau\colon T\to S $ a finite, purely inseparable morphism of height one.
		Assume that $X_{K(T)}$ is integral.
		Consider the commutative diagram
		$$\xymatrix{Y:=(X_T)^\nu  \ar[r]^>>>>\pi\ar[d]^g &X\ar[d]^{f}\\
			T\ar[r]^{\tau} &S\rlap. \\
		}$$
		Assume moreover that either $f$ is equidimensional or that $T,S$ are $\mathbb Q$-Gorenstein,
		then there exists an effective divisor $E$ and a $g$-exceptional $\mathbb Q$-divisor $N$ on $Y$ such that
		\begin{equation}\label{eq:KQSD}
			\pi^*K_{X/S} \simQ K_{Y/T} + (p-1)E +N.
		\end{equation}
		\end{prop}
		\begin{proof}
		To prove the assertion, we may restrict ourselves to the regular locus of $X, Y$.
		So we may assume that $X, Y$ are both regular.
		Let $S^\circ$ be a big regular open subset such that $T^\circ := \tau^{-1}S^\circ$ is regular.
		Let $X^\circ = X_{S^\circ}$ and $Y^\circ=Y_{T^\circ}$.
		Since $X_{K(T)}$ is assumed to be integral, the natural homomorphism $\delta\colon  g^*\Omega_{T^\circ/S^\circ}^1 \to \Omega_{Y^\circ/X^\circ}^1$ is injective and has the same rank, hence induces an injective homomorphism $\det(\delta)\colon  g^*\det \Omega_{T^\circ/S^\circ}^1 \to \det \Omega_{Y^\circ/X^\circ}^1$.
		Therefore we may identify $g^*\det \Omega_{T^\circ/S^\circ}^1 \cong\det \Omega_{Y^\circ/X^\circ}^1(-E_0)$ for some effective divisor $E_0$ on $Y^\circ$.
		Applying Proposition~\ref{prop:can-foliation}, we have
		\begin{align*}
			\pi^*K_{X^\circ} &\sim K_{Y^\circ} + (p-1)\det \Omega_{Y^\circ/X^\circ}^1 \sim K_{Y^\circ} + (p-1)g^*\det \Omega_{T^\circ/S^\circ}^1 + (p-1)E_0 \\
			&\sim K_{Y^\circ} + g^*(\tau^*K_{S^\circ} - K_{T^\circ}) + (p-1)E_0,
		\end{align*}
		which gives
		$$\pi^*K_{X^\circ/S^\circ} \sim K_{Y^\circ/T^\circ} + (p-1)E_0.$$
		Let $E$ be the closure of $E_0$ on $Y$. We may extend the above relation to the whole variety $Y$ up to some $g$-exceptional $\mathbb Q$-divisor $N$, which is the relation (\ref{eq:KQSD}) as desired.
		\end{proof}
		
		\begin{rmk}\label{rmk:E}
		It is worth mentioning that the restriction of $E$ on the generic fiber $Y_{K(T)}$ of $g$ coincides with the conductor of the normalization 
		$\nu\colon Y\to X_T$.
		In particular, if $X_{K(T)}$ is normal, then $E$ is $g$-vertical.
		\end{rmk}

		\section{Curves of Arithmetic genus one}\label{sec:g=1}
		This section focuses on regular but non-smooth projective curves of arithmetic genus one defined over an imperfect field.
		We will look closely at the behavior of the non-smooth locus under height-one base changes.
		
		\subsection{Auxiliary results}
		First, we may deduce the following result from \cite[Subsection~3.2.2, Corollary~2.14 and Proposition~2.15]{Liu02}.
		\begin{prop}\label{prop:basic-fib}
		Let $f\colon  X \to S$ be a fibration of normal varieties over a field $k$ of characteristic $p>0$. Then $(X_{\overline{K(S)}})_{\mathrm{red}}$ is integral, and the fibration $f$ is separable if and only if the geometric generic fiber $X_{\overline{K(S)}}$ is reduced.
		\end{prop}

		We now consider a curve $X$ over a field $K$.
		To be precise, by a \emph{curve} over $K$ we mean a purely one-dimensional quasi-projective scheme over $K$.
		Let $D =\sum_i a_i \mathfrak p_i$ be a Cartier divisor on $X$.
		The {\em degree} of $D$ is defined to be the integer \[
		\deg_K D= \sum_i a_i [\kappa(\mathfrak p_i):K].
		\]
		where $\kappa(\mathfrak p_i)$ denotes the residue field of $\mathfrak p_i$.

		We will need the following classification of curves of arithmetic genus zero.
		\begin{prop}[{\cite[Theorem~9.10]{Tanaka21}}]\label{prop:ga0}
		Let $X$ be a normal projective integral $K$-curve with
		$H^0(X, \mathcal O_X)=K$ and $H^1(X, \mathcal O_X)=0$.  Then the following statements hold.
		\begin{enumerate}[\rm(1)]
			\item $\deg_K K_X = -2$.
			\item $X$ is isomorphic to a conic in $\mathbb{P}^2_K$ and $X \cong \mathbb{P}^1_K$ if and only if it has a $K$-rational point.
			\item Either $X$ is a smooth conic or $X$ is geometrically non-reduced.
			In the latter case, we have $\Char K =2$, and $X$ is isomorphic to the curve defined by a quadric $sx^2 + ty^2 + z^2 = 0$ for some $s,t\in K\setminus K^2$.
		\end{enumerate}
		\end{prop}

		\subsection{Notation and assumptions}\label{sec:curve-notation}
		From now on to the end of this section, we work over a field $K$ of characteristic $p>0$ such that $[K:K^p] < \infty$.
		Denote by $\overline K$ the algebraic closure of $K$.
		Let $X$ be a regular projective curve over $K$ with $H^0(X,\mathcal O_{X}) = K$.
		Assume that $X$ has arithmetic genus one,  i.e., $h^1(X,\mathcal O_X) = 1$.
		Assume further that $X$ is not smooth over $K$.
		Then there exists an intermediate field $L$ with $K\subset L \subseteq K^{1/p}$ such that $X_L := X\otimes_K L$ is integral but not regular (\cite[Proposition~1.5]{Sch10}).
		We fix such an $L$.
		Note that if $L/K$ is inseparable of degree $p$, then $X_L$ is always integral.
		If $Y$ denotes the normalization of $X_L$ and $\pi\colon Y\to X$ the induced morphism, then $K':=H^0(Y,\mathcal O_Y)\subseteq K^{1/p}$.
		We have the following commutative diagram:
		\begin{equation}\label{eq:4A2P}
		\vcenter{\xymatrix{&X^{\oneoverp} \ar[d] \ar[rrd]^{F_X} & & \\
				&Y:=(X_L)^\nu \ar[d]\ar[r] &X_L\ar[r]\ar[d] &X\ar[d]\\
				&\Spec K' \ar[r]
				&\Spec L\ar[r] &\Spec K. \\
				}}\end{equation}
				Remark that if $X$ is geometrically reduced, then $K' = L$.
				For this, we adapt the proof of \cite[Theorem~3.1 (2)]{Tanaka21} to our situation. 
				Since there is an injective ring homomorphism $K' \hookrightarrow K(X_L)$ and the scheme $X_{K'}$ is integral, the induced homomorphism $K' \otimes_L K' \hookrightarrow K(X_L)\otimes_L K' = K(X_{K'})$ is injective.
				Then $K' \otimes_L K'$ is reduced, which implies $K' = L$ because $L \subset K'$ is purely inseparable.
				
				Recall that \begin{equation}
		\label{eq:8Q12}
		\pi^*K_X \sim K_Y + (p-1) C,
		\end{equation}
		where $C>0$ is supported on the inverse image of the non-normal point of $X_L$ and $(p-1) C$ coincides with the usual conductor divisor
		(\cite[Theorem~1.2]{PW22}).
		
		\begin{prop}\label{prop:basic}
		With the setting above,
		\begin{enumerate}[\rm(1)]
			\item the characteristic $p$ equals $2$ or $3$, and the normalization of $X_{\overline{K},\mathrm{red}}$ is isomorphic to $\mathbb{P}^1_{\overline{K}}$;
			\item if $X$ is geometrically reduced, then $X$ is geometrically integral, and there exists a unique singular point on $X_{\overline{K}}$;
			\item if $X$ is geometrically non-reduced, then either $X_{\overline{K},\mathrm{red}}= \mathbb P^1_{\overline{K}}$ or $X_{\overline{K},\mathrm{red}}$ has a unique singular point.
		\end{enumerate}
		\end{prop}
		\begin{proof}
		(1) 
		By (\ref{eq:8Q12}) we have $K_Y \sim -(p-1) C < 0$. By Proposition~\ref{prop:ga0}, we see that $Y$ is a conic (or just $\mathbb P^1$) and $\deg_{K'} K_Y = -2$.
		We conclude that $p=2$ or $3$ and that $(X_{\overline K})_{\rm red}^\nu \cong \mathbb P^1_{\overline K}$.
		
		By Proposition~\ref{prop:basic-fib}, $X_{\overline K,\rm red}$ is integral. To show (2) and (3), it suffices to show that $X_{\overline K,\rm red}$ has at most one singular point, or more strongly, $p_a(X_{\overline K,\rm red}) \le 1$.
		
		If $X$ is geometrically reduced, then $p_a(X_{\overline K}) = h^1(X_{\overline K},\mathcal O_{X_{\overline K}}) = 1$. 
		If $X$ is geometrically non-reduced, and if we denote by $\mathcal N$ the nilradical ideal sheaf of $X_{\overline K,\rm red}$, then by the exact sequence 
		$0\to \mathcal N \to \mathcal O_{X_{\overline K}} \to \mathcal O_{X_{\overline K,\rm red}} \to 0$,
		we see that $p_a(X_{\overline{K},\mathrm{red}}) \le 1$.
		\end{proof}
		
		In the following we consider the behavior of pulling back divisors.
		Assume for now that $K\subset L$ is a finite (not necessarily purely inseparable) field extension.
		Let $\mathfrak p$ be a closed point on $X$, and let $\mathfrak q_1,\ldots,\mathfrak q_r$ be the points on $Y$ lying over $\mathfrak p$.
		We write \[
		\pi^*\mathfrak p = e_1\mathfrak q_1 +\cdots+ e_r \mathfrak q_r,
		\]
		where $e_i$ are the ramification index.
		Let $f_i := [\kappa(\mathfrak q_i):\kappa(\mathfrak p)]$ denote the residue class degree.
		It is well known that (cf.\ \cite[Proposition~7.1.38]{Liu02})
		\begin{equation}\label{eq:1BJT}
		[L:K] = \sum_i e_i f_i.
		\end{equation}
		
		Back to our initial setting, where $K\subset L$ is purely inseparable of height one, we have that $Y\to X$ is homeomorphic, and thus \[
		\pi^*\mathfrak p = p^\gamma \mathfrak q,
		\]
		for some integer $\gamma$.
		Since $F_X^* \mathfrak p = p\mathfrak p$, by the diagram (\ref{eq:4A2P}) we see that $\gamma = 0$ or $1$.

		\begin{examp}\label{exam:conic-pull-back}
		Let $X/K$ be the non-smooth regular curve defined by the affine equation $s x^2 + t y^2 + 1 = 0$,
		where $\mathop{\rm char} K = 2$ and $s,t\in K\setminus K^2$.
		Let $\mathfrak p\in X$ be the prime ideal $(y)$.
		Then $\kappa(\mathfrak p) = K(s^{1/2}) =: L$.
		Now $X\times_K L = \Spec L[x,y]/((s^{1/2}x + 1)^2 + ty^2)$ is an integral curve.
		Let $\pi\colon Y \to X\times_K L$ be the normalization morphism:
		\[\tabskip=5pt\def\mapsby{\gets\joinrel\mapstochar}
		\vbox{\halign{\hfil$#$&$#$&$#$\cr
				Y \cong \Spec K(s^{1/2},t^{1/2})[y] &\to&     X\times_KL\cr
				y                                   &\mapsby& y\cr
				(t^{1/2} y + 1)/ s^{1/2}            &\mapsby& x\cr
		}}
		\]
		We have $H^0(Y,\mathcal O_{Y}) = K(s^{1/2}, t^{1/2}) =: K'$.
		If we denote by $\mathfrak q\in Y$ the prime ideal $(y)\subset K'[y]$,
		then $\mathfrak q$ is a $K'$-rational point and $\pi^* \mathfrak p = \mathfrak q$,
		as predicted by (\ref{eq:1BJT}).
		\end{examp}
		
		\subsection{The case when $X$ is geometrically reduced.}\label{geom-integral}
		\begin{prop}\label{prop:ga1-geo-reduced-p=2}
		With the notation and assumptions as in subsection~\ref{sec:curve-notation}, assume that $X$ is geometrically reduced, then:
		\begin{enumerate}[\rm(1)]
			\item
			The non-smooth locus of $X$ is supported at a closed point $\mathfrak p$ and
			$\kappa(\mathfrak p)/K$ is purely inseparable of height one with $[\kappa(\mathfrak p):K]\le p^2$.
			Let $\mathfrak q \in Y$ be the unique point lying over $\mathfrak p$.
			\item
			If $p=3$, then $Y\cong\mathbb{P}^1_L$, $\pi^*\mathfrak p=3\mathfrak q$, $\mathfrak q$ is an $L$-rational point, and $[\kappa(\mathfrak p):K]=3$.
			\item Assume $p=2$.
			\begin{enumerate}[\rm(a),leftmargin=1.5em]
				\item If the point $\mathfrak q$ is $L$-rational, then $Y\cong \mathbb{P}^1_L$, $\pi^*\mathfrak p=2\mathfrak q$ and $[\kappa(\mathfrak p):K]=2$.
				\item If the point $\mathfrak q$ is not $L$-rational, then $\deg_L \mathfrak q=2$, and
				\[
				\pi^*\mathfrak p=
				\begin{cases}
					\mathfrak q,  &\hbox{ if } \deg_K(\mathfrak p) = 2;\\
					2\mathfrak q, &\hbox{ if } \deg_K(\mathfrak p) = 4.
				\end{cases}\]
			\end{enumerate}
		\end{enumerate}
		\end{prop}
		\begin{proof}
		(1) See Lemma~1 and Theorem~2 in \cite{Queen71}.
		
		(2) As noted before, since $X$ is geometrically reduced, we have $L = K' = H^0(Y,\mathcal O_Y)$.
		Then, by $\deg_{L}(K_Y) = -2$ and $K_Y + (p-1)C \sim 0$ where $C$ is supported at $\mathfrak q$, we see that $\mathfrak q$ is an $L$-rational point.
		Hence $Y\cong \mathbb P^1_L$. Assume $\pi^* \mathfrak p = p^\gamma \mathfrak q$, where $\gamma = 0$ or $1$.
		By (\ref{eq:1BJT}), we have $[L:K] = p^\gamma f = p^\gamma[L:\kappa(\mathfrak p)]$, and consequently $p^\gamma  = [\kappa(\mathfrak p):K]$.
		Note that $\mathfrak p$ is not a $K$-rational point (cf.\ \cite[Proposition~2.13]{Tanaka21}), thus $\gamma = 1$.
		
		(3) When $\mathfrak q$ is an $L$-rational point, we use the proof of (2) to obtain (a).
		When $\mathfrak q$ is not $L$-rational, we have $\deg_L \mathfrak q = 2$ by Proposition~\ref{prop:ga0}
		and $\deg_K \mathfrak p = 2$ or $4$ by (1).
		Now $[L:K] = p^\gamma f = 2 p^\gamma [L:\kappa(\mathfrak p)]$.
		The statement follows.
		\end{proof}

		We now give an explicit example of Proposition~\ref{prop:ga1-geo-reduced-p=2}.
		\begin{examp}%
		Let $K$ be a field of characteristic $2$ with $a,b\in K\setminus K^2$.
		Consider the following affine regular curve \begin{equation}
			\label{eq:example-geo-reduced-curve}
			X : y^2 = x^3 + ax + b.
		\end{equation}
		Let $\alpha =a^{1/2}$, $\beta = b^{1/2}$ and $L = k(\alpha,\beta)$.
		Then $X_L$ has a singular point $\mathfrak p_L := (x=\alpha, y=\beta)$.
		Let $Y=(X_L)^\nu$ be the normalization.
		It is easy to see that $t := (y+\beta)/(x+\alpha)$ is a local parameter of the point of $Y$ lying over $\mathfrak p_L$,
		thus $Y \cong \mathbb A^1_{k(\alpha,\beta)}$.
		Besides, the normalization morphism is given by \[
		x \mapsto t^2,\quad y \mapsto t^3 + \alpha t + \beta.
		\]
		Now, the point $\mathfrak p\in X$ defined by the prime ideal $(x^2 + a)$ has residue field $K[x,y]/(x^2+a, y^2+b)\cong k(\alpha,\beta)$,
		which has degree $4$ over $K$.
		The pullback of $\mathfrak p$ becomes the ideal $(t^4+\alpha^2) = (t^2+\alpha)^2$.
		If $\mathfrak q\in Y$ is the point corresponding to $(t^2 + \alpha)$, then \[
		\pi^* \mathfrak p = 2\mathfrak q \text{ \ and \ }  [\kappa(\mathfrak q):k(\alpha,\beta)] = 2.
		\]
		Therefore the completion $\overline X$ of $X$ is an example of case~(ii) in Proposition~\ref{prop:ga1-geo-reduced-p=2}.
		By replacing $a$ (resp.\ $b$) by $0$ in (\ref{eq:example-geo-reduced-curve}), we obtain an example for
		case~(a) (resp.\ (i)) in Proposition~\ref{prop:ga1-geo-reduced-p=2}.
		\end{examp}

		\subsection{The case when $X$ is geometrically non-reduced}
		With the notation and assumptions as in subsection~\ref{sec:curve-notation}, assume further that $X$ is geometrically non-reduced.
		There exists, as mentioned before, a height one field extension $K\subset L$ such that $X_L$ is integral but not normal, and for the normalization $Y$ of $X_L$, we have $L \subsetneq K' := H^0(Y,\mathcal O_Y)\subseteq K^{\frac{1}{p}}$.
		Since $0\sim \pi^* K_X \sim K_Y + (p-1)C$,  we see that $\deg_{K'} (p-1)C = 2$.
		Thus $C$ is supported on either a single point $\mathfrak q\in Y$ or two points $\mathfrak q_1,\mathfrak q_2\in Y$ (this happens only when $p=2$).
		We list all the possibilities explicitly in the following without proof as it is straightforward.
		
		\begin{prop}\label{prop:curve-nonreduced}
		\begin{enumerate}[\rm(1)]
			\item If we denote $X' := (X_{K'})_{\rm red}$, then $Y \cong (X')^\nu$.
			\item
			If $p=3$, then $X_L$ has a unique non-normal point, $ C= \mathfrak q$,  $Y\cong \mathbb{P}_{K'}^1$ and either $\pi^*\mathfrak p = \mathfrak q$ or $\pi^*\mathfrak p = 3\mathfrak q$.
			\item
			If $p=2$, then we fall into one of the following cases:
			\begin{enumerate}[\rm a),leftmargin=1.5em]
				\item $ C= 2\mathfrak q$, thus $\mathfrak q$ is a $K'$-rational point of $Y$ and $Y \cong \mathbb{P}_{K'}^1$;
				\item $ C= \mathfrak q_1+\mathfrak q_2$, and also $Y \cong \mathbb{P}_{K'}^1$;
				\item $ C= \mathfrak q$, $\kappa(\mathfrak q)/K'$ is an extension of degree two, and either
				\begin{enumerate}[\rm c1),leftmargin=1.5em]
					\item $Y\subset\mathbb{P}^2_{K'}$ is a smooth conic (possibly $\mathbb P^1_{K'}$), or
					\item $Y$ is isomorphic to the curve defined by $sx^2 + ty^2 + z^2 = 0$ for some $s,t\in K'\setminus K'^2$ such that $[K'^2(s,t):K'^2]=4$.
				\end{enumerate}
			\end{enumerate}
		\end{enumerate}
		\end{prop}

		\section{Canonical bundle formula for fibrations with generic fiber of arithmetic genus zero}
		In this section, we shall treat a special kind of fibration with the generic fiber being a curve of arithmetic genus zero, which is an intermediate situation when treating inseparable fibrations. We work over an algebraically closed field $k$ of characteristic $p>0$.
		
		\begin{thm}\label{thm:can-bd-g0}
		Let $f\colon X\to S$ be a fibration of relative dimension one from a normal $\mathbb Q$-factorial quasi-projective variety $X$ onto a normal variety $S$.
		Let $\Delta$ be an effective $\mathbb Q$-divisor on $X$ and $D$ a $\mathbb Q$-Cartier $\mathbb Q$-divisor on $S$.
		Let $\mathfrak M$ be a movable linear system without fixed components such that
		\[\deg_{K(S)} \mathfrak M:=\deg_{K(S)} M_0|_{X_{K(S)}} > 0 \hbox{ \ for some \ } M_0\in \mathfrak M.\]
		Assume that $K_X + M_0 + \Delta \sim_\mathbb Q f^* D$ and that one of the following conditions holds:
		\begin{enumerate}[\rm(a)]
			\item $f$ is separable;
			\item $f$ is inseparable, and there is a generically finite morphism from $S$ to an abelian variety;
			\item $f$ is inseparable, and $\dim S=2$.
		\end{enumerate}
		Then there exists a rational number $t>0$ such that $D\geQ tK_S$, where we can take $t=1$ {\rm(}resp.\ $1/2$, $3/4${\rm)} under the condition {\rm(a)} {\rm(}resp.\ {\rm(b), (c))}.
		\end{thm}
		\begin{proof}
		Since the statement involves only codimension one points, we may restrict to a big open subset of $S$ and assume that $S$ is regular.
		
		By assumption we have $p_a(X_{K(S)})=0$, thus $\deg_{K(S)} (M_0+\Delta) = 2$.
		In the following, we take $M\in \mathfrak M$ to be a general member and $T$ one of its irreducible horizontal components. Write $M= T + G$.
		We get $G|_{T^\nu} \geQ 0$ by the following lemma.
		\begin{lem}\label{lem:4K52}%
			Let $X$ be a normal $\mathbb Q$-factorial variety and $M$ a movable divisor on $X$.
			If $T$ is an irreducible component of $M$ that is not a fixed component, then $(M-T)|_{T^\nu} \geQ 0$.
		\end{lem}
		\begin{proof}
			Write $M = nT + G'$ with $T\wedge G' = 0$. Then $(M-T)|_{T^\nu} \simQ \frac1n \bigl( (n-1)M + G'\bigr)|_{T^\nu} \geQ 0$.
		\end{proof}
		
		\smallskip
		
		Write $\Delta=\alpha T + \Delta'$ with $T\wedge \Delta' = 0$, and then $0\le \alpha \le 1$ as $\deg_{K(S)} (M_0+\Delta) = 2$. 
		We may write
		\[
		G+\Delta = G+\alpha T+\Delta' = (1-\alpha)G+\alpha M +\Delta',
		\]
		thus $(G+\Delta)|_{T^\nu}  \succeq_{\mathbb{Q}} 0$.

		\medskip
		Case\thinspace (1): $\deg_{K(S)} T =1$.
		Then $T \to S$ is a birational section. By adjunction formula, there exists an effective divisor $E_{T^\nu}$ on $T^\nu$ such that
		$$(K_X+ M+\Delta)|_{T^\nu} = (K_X + T + G + \Delta)|_{T^\nu} \sim_{\mathbb{Q}} K_{T^\nu} + E_{T^\nu}.$$
		Denote by $\sigma\colon T^\nu \to S$ the natural birational morphism.
		Now we have $K_{T^\nu} + E_{T^\nu} \sim_{\mathbb{Q}} \sigma^*D$, and then by Lemma~\ref{lem:push-down}
		we obtain
		$$D \sim_{\mathbb{Q}} \sigma_*(K_{T^\nu} + E_{T^\nu}) = K_S + E_S,$$
		where $E_S \geq 0$.
		
		\medskip
		Case\thinspace (2): $\deg_{K(S)} T =2$.
		
		Case\thinspace (2.1): $T \to S$ is a separable morphism.
		As in case~(1), there exists $E_{T^\nu} \geq0$ such that
		\begin{equation}\label{eq:82GZ}
			\sigma^*D\sim_{\mathbb{Q}} (K_X+ M + \Delta)|_{T^\nu}\sim_{\mathbb{Q}} K_{T^\nu} + E_{T^\nu}.
		\end{equation}
		Since $\sigma\colon T^\nu\to S$ is separable, $K_{T^\nu} \succeq \sigma_*K_S+R_{T^\nu}+N_{T^\nu}$, where $R_{T^\nu}\ge0$ and $N_{T^\nu}$ is exceptional.
		Now, pushing forward (\ref{eq:82GZ}) via $\sigma$ gives
		$$D \sim_{\mathbb{Q}} \frac12\sigma_*(K_{T^\nu} + E_{T^\nu}) \geQ  K_S.$$
		
		\medskip
		Case\thinspace (2.2): $T \to S$ is inseparable, which only happens when $p=2$. Let $S_1$ be the normalization of $S$ in $K(T)$. 
		In other words,  the natural morphisms $T^\nu \to S_1 \to S$ form the Stein factorization of $\sigma$. Hence, $T^\nu\to S_1$ is birational and $S_1\to S$ is finite.
		Since $K(S_1)/K(S)$ is a purely inseparable extension of degree two, $X_{K(S_1)}$ is integral by Proposition~\ref{prop:ga0}.
		By shrinking $S$ again, we may assume that $K_{S_1}$ is $\mathbb Q$-Cartier.
		
		\medskip
		Case\thinspace (2.2.1): $X_{K(S_1)}$ is normal. Consider the following commutative diagram
		$$\xymatrix{
			&X_1:=(X_{S_1})^\nu \ar[rd]^g\ar[r]^<<<<{\nu}\ar@/^7mm/[rr]|{\;\pi\;} &X_{S_1}\ar[r]\ar[d] &X\ar[d]^{f}\\
			&   &S_1\ar[r]^{\tau} &S\rlap. \\
		}$$
		The morphism $\pi\colon X_1 \to X$ is finite and purely inseparable, and the normalization morphism $\nu\colon X_1 \to X_{S_1}$ induces an isomorphism of the generic fibers over $S_1$. Since $X_{S_1} \to S_1$ has a birational section which is mapped to $T$, the prime divisor $T_1$ supported on $\pi^{-1}T$ is a birational section, and $\pi^*T = 2T_1$.
		Applying Proposition~\ref{prop:compds} and Remark~\ref{rmk:E}, we have
		\begin{equation}\label{eq:4ZHM}
			\pi^*(K_X - f^*K_S) \sim_{\mathbb{Q}} K_{X_1} - g^*K_{S_1} + E_1 + N,
		\end{equation}
		where $E_1$ is effective and $g$-vertical, and $N$ is exceptional over $S$.
		Write $\pi^*M = 2T_1 + E_1'$. Then
		$$\pi^*(K_X+ M+\Delta) \simQ K_{X_1} + 2T_1 + g^*\tau^*K_S - g^*K_{S_1} + E_1 + \pi^*\Delta + E_1' + N.$$
		Denote by $\sigma_1\colon T^\nu_1 \buildrel\nu\over\to T_1 \to S_1$ the composition morphism. Consider the restriction on $T^\nu_1$. Applying the adjunction formula $(K_{X_1} + T_1)|_{T^\nu_1} \sim_{\mathbb{Q}} K_{T^\nu_1} + \Delta_{T^\nu_1}$, we have
		\begin{equation}\label{eq:theo-5.1-2.2.1-5}
			\begin{split}
				\tau^*D|_{T^\nu_1} &\sim_{\mathbb{Q}}(K_{X_1} + T_1 + T_1 + g^*\tau^*K_S - g^*K_{S_1}  +E_1 + \pi^*\Delta + E_1' + N)|_{T^\nu_1}\\
				& \sim_{\mathbb{Q}} (K_{T^\nu_1} - \sigma_1^*K_{S_1}) + \Delta_{T^\nu_1} +(T_1+E_1' + E_1+\pi^*\Delta)|_{T^\nu_1}+ \sigma_1^*\tau^*K_S + N|_{T^\nu_1}.
			\end{split}
		\end{equation}
		Note that
		\begin{itemize}
			\item
			$(T_1+E_1')|_{T^\nu_1}\geQ 0$ by Lemma~\ref{lem:4K52} since $\pi^*M = 2T_1 + E_1'$;
			\item
			both $K_{T^\nu_1} - \sigma_1^*K_{S_1}$ and $N|_{T^\nu_1}$ are exceptional over $S$;
			\item both $E_1$ and $\Delta$ are vertical over $S$, and consequently, the restriction $(E_1+\pi^*\Delta)|_{T_1^\nu}$ is effective.
		\end{itemize}
		Applying $\tau_*\sigma_{1*}$ to (\ref{eq:theo-5.1-2.2.1-5}), by Lemma~\ref{lem:push-down} we obtain that $D \sim_{\mathbb{Q}}  K_S + E_S$ for some divisor $E_S \geq 0$.
		
		\medskip
		Case\thinspace (2.2.2): $X_{K(S_1)}$ is not normal. This means that $X_{K(S)}$ is a non-smooth conic as described in Proposition~\ref{prop:ga0} (3). In this case, $f$ is inseparable and we assume that one of the conditions (b,\thinspace c) holds.
		
		Applying Proposition~\ref{prop:ga0}, we see that $X_{K(S_1)}$ is not normal along the preimage of the generic point of $T$, and $K(S_1)$ is not algebraically closed in $K(X)\otimes_{K(S)}K(S_1)$. Let $X_1:=(X_{S_1})^\nu$ and denote by $S_1'$ the normalization of $S_1$ in $X_1$.
		Let $T_1$ be the irreducible divisor corresponding to $\pi^{-1}T$.
		Denote by $\rho\colon T^\nu_1 \to T^\nu$ the induced morphism by the normalizations.
		These varieties fit into the following commutative diagram
		$$\xymatrix{
			&T^\nu_1\ar[d]\ar[rr]^{\rho} & & T^\nu\ar[d]\ar[ldd]_<<<<<<{\delta}\\
			&X_1:=(X_{S_1})^\nu \ar[d]^g\ar[r]^<<<<{\nu}\ar@/^2pc/[rr]^{\pi} &X_{S_1}\ar[r]_>>>{\pi_1}\ar[d]_{f_1} &X\ar[d]^{f}\\
			&S_1'\ar[r]   &S_1\ar[r]^{\tau} &S\rlap. \\
		}$$
		The conductor divisor of $X_1 \to X_{S_1}$ is like $aT_1 + C'$ where $a\geq 1$, namely, $\nu^*K_{X_{S_1}}= K_{X_1} + aT_1 + C'$.
		Similarly to (\ref{eq:4ZHM}), we have
		$$\pi^*K_X \simQ K_{X_1} + aT_1 + E_1 + N_1 - \nu^*f_1^*K_{S_1/S} ,$$
		where $E_1$ is effective and $N_1$ is $g$-vertical.
		We may write that $\pi^*T = bT_1$ and $\pi^*M = bT_1 + E_1'$. Then
		$$\pi^*(K_X+ M+\Delta) \simQ K_{X_1} + (a+b)T_1 + E_2 + N_1 - \nu^*f_1^*K_{S_1/S},$$
		where $E_2 = E_1 + \pi^*\Delta + E'_1$ is effective.
		Since $(K_X+ M+\Delta)_{K(S)} \sim 0$, we conclude that $a=b=1$ and that $T_1 \to S'_1$ is birational.
		Note that $T\not\subseteq\Supp\Delta$, thus Lemma~\ref{lem:log-adj} gives
		$$T_1|_{T^\nu_1} = \rho^*(T|_{T^\nu}) \simQ \rho^*(K_{T^\nu} + B_{T^\nu} + M|_{T^\nu}) - \pi^*(K_X+M+\Delta)|_{T^\nu_1},$$
		where $B_{T^\nu}\ge0$.
		Applying the adjunction formula on $T_1$ we have
		\begin{equation*}
			\begin{split}
				\pi^*(K_X+ M+\Delta)|_{T^\nu_1} & \simQ (K_{X_1} + T_1)|_{T^\nu_1} +  T_1|_{T^\nu_1} + (E_2 + N_1 - \nu^*f_1^*K_{S_1/S})|_{T^\nu_1} \\
				&\sim_{\mathbb{Q}} K_{T^\nu_1} + \Delta_{T^\nu_1} + \rho^*(K_{T^\nu} + B_{T^\nu}+M|_{T^\nu}) - \pi^*(K_X+M+\Delta)|_{T^\nu_1} \\
				& \qquad + (E_2 + N_1 - \nu^*f_1^*K_{S_1/S})|_{T^\nu_1}
			\end{split}
		\end{equation*}
		It follows that
		\begin{equation}\label{eq:restr}
			\begin{split}
				2\pi^*(K_X+ M+\Delta)|_{T^\nu_1} \sim_{\mathbb{Q}} {}& (\pi^*f^*K_S)|_{T_1^\nu} + K_{T^\nu_1}  \\
				&+ \Delta_{T^\nu_1} + \rho^*B_{T^\nu} + \pi^*M|_{T^\nu_1}+ E_2|_{T^\nu_1} + N_2
			\end{split}
		\end{equation}
		where $N_2= N_1|_{T^\nu_1} + \rho^*(K_{T^\nu}-\delta^*K_{S_1})$ is exceptional over $S$.
		
		If condition (b) holds, then there is a generically finite morphism $T_1^\nu\to A$, where $A$ is an abelian variety. 
		Using Nagata's compactification over $A$, we can embed  $T_1^\nu$ into a proper and normal birational model $\overline T_1$ over $A$.
		Since $T_1^\nu$ is quasi-projective, by Chow's lemma (\cite[Corollary~2.6]{Conrad07}), we may assume that $\overline T_1$ is projective. 
		Since $\overline T_1$ admits a generically finite morphism to an abelian variety, $\overline T_1$ is of m.A.d.
		Applying Proposition~\ref{prop:char-abel-var}, we have that $K_{\overline{T}_1} \geQ 0$, and thus $K_{T_1^\nu} \geQ 0$.
		Note that $2\pi^*(K_X+ M+\Delta)|_{T^\nu_1} = 2\tau^*D|_{T^\nu_1}$. Pushing down the equation (\ref{eq:restr}) via the morphism $T^\nu_1 \to S$ yields
		$D \sim_{\mathbb{Q}}  \frac{1}{2}K_S + E_S$ for some divisor $E_S \geq 0$.

		If the condition (c) holds, then $S_1'\to S$ is of height one and of degree $\geq p^2=4$, this implies $S_1' = S^{\oneoverp}$. It follows that $K_{T^\nu_1} = g^*K_{S^{\oneoverp}}|_{T^\nu_1}  + N_3$ where $N_3$ is a divisor on $T^\nu_1$ exceptional over $S$.
		Pushing down the equation (\ref{eq:restr}) via the morphism $T^\nu_1 \to S$ yields
		$D \sim_{\mathbb{Q}}  \frac{3}{4}K_S + E_S$ for some divisor $E_S \geq 0$.
		\end{proof}
		\medskip

		As an application of the above theorem, if $K_X + M_0 + \Delta$ is anti-nef, we obtain the following structure result, which will play a key role in the proof of Theorem~\ref{thm:intro:S-abelian}.
		\begin{prop}\label{prop:str-num0}
		Let $X$ be a $\mathbb{Q}$-factorial normal projective variety, and let $\Delta$ be an effective $\mathbb Q$-divisor on $X$.
		Let $f\colon X\to S$ be a fibration of relative dimension one,
		and $\mathfrak M$ a movable linear system without fixed components such that
		$\deg_{K(S)} \mathfrak M >0$.
		Assume that the following three conditions hold:
		\begin{itemize}
			\item[\rm(1)] $S$ is of m.A.d.;
			\item[\rm(2)] $-(K_X + M_0 + \Delta)$ is nef, where $M_0\in\mathfrak M$; and
			\item[\rm(3)]
			either $(X_{K(S)}, \Delta_{K(S)})$ is klt, or if $T$ is a (the unique) horizontal irreducible component of $\Delta$ with coefficient one, then $\deg_{K(S)} T =1$ and the restriction $T|_{T^\nu}$ on the normalization of $T$ is pseudo-effective.
		\end{itemize}
		Then
		\begin{itemize}
			\item[\rm(i)]   $S$ is isomorphic to an abelian variety;
			\item[\rm(ii)]  $M_0$ is semi-ample with numerical dimension $\nu(M_0) =1$;
			\item[\rm(iii)] $|M_0|$ induces a fibration $g\colon X\to \mathbb{P}^1$;
			\item[\rm(iv)]  
			a fiber of $g$ over a closed point $t\in\mathbb P^1$ (denoted by $G_t$) is either isomorphic to an abelian variety or a multiple of an abelian variety, moreover a general fiber $G_t$ is reduced and $\Delta|_{G_t} = 0$.
			
		\end{itemize}
		\end{prop}
		\begin{proof}
		Let $M \in \mathfrak M$ be a general divisor and write $M = T + M' + V$ where $T$ is a horizontal irreducible component, $V$ is the vertical part, and $M'$ is the remaining part.
		Since $\deg_{K(S)} M_0 \leq 2$, we have $M'=0$, $M'=T$, or $M'$ is another horizontal component.
		Since $T$ is dominant over $S$, it is of m.A.d., and we have $K_{T^\nu} \geQ 0$.
		We have $(M'+ V)|_{T^\nu} \geQ 0$ by Lemma~\ref{lem:4K52}.
		
		\smallskip
		For convenience, we first establish a lemma.
		\begin{lem}\label{lem:V=0} With the notation above, we have $V=0$, $M'|_{T^\nu} \simQ \Delta|_{T^\nu}\simQ  0$, $(K_{X} + M +\Delta)|_{T^\nu} \equiv 0$, and $T = T^\nu$ is isomorphic to an abelian variety.
		\end{lem}
		\begin{proof}[Proof of the lemma]
			We first remark that
			\begin{itemize}
				\item[$(*)$]
				Suppose $V \neq 0$, as $M$ varies in $\mathfrak M$ we get a family of numerically equivalent horizontal divisors $\mathfrak T$ and a family of numerically equivalent vertical divisors $\mathfrak V$, which contains $T$ and $V$ respectively. Both of the families of divisors cover $X$, therefore, $\Supp T \cap \Supp V \neq \emptyset$.
			\end{itemize}
			By the adjunction formula, we have
			\begin{equation}\label{eq:WLTT}
				(K_{X} + M +\Delta)|_{T^\nu} = (K_{X} + T+ (M'+V)+ \Delta)|_{T^\nu} \simQ K_{T^\nu} + \Delta_{T^\nu} + (M'+V)|_{T^\nu} + \Delta|_{T^\nu},
			\end{equation}
			which is anti-nef.
			Since
			$K_{T^\nu}, \Delta_{T^\nu}, (M'+ V)|_{T^\nu}$
			and
			$\Delta|_{T^\nu}$ (note that $T\wedge \Delta=0$ for a general $M$)
			are $\mathbb Q$-effective, we see that
			$$(K_X + M+\Delta)|_{T^\nu} \equiv 0,\quad K_{T^\nu} \simQ (M' + V)|_{T^\nu}\simQ 0 \text{ \ and \ } \Delta_{T^\nu}  = \Delta|_{T^\nu} = 0.$$
			
			We show that $V|_{T^\nu} \simQ 0$ by distinguishing the following two cases:
			\begin{itemize}
				\item[1.] If $M'\neq T$, then $M'|_{T^\nu} \geQ 0$ and thus $V|_{T^\nu} \simQ 0$.
				\item[2.] If $M'=T$, we can take another element $M_1\in|M_0|$ so that $T \not\subseteq M_1$. Then by
				\[
				0 \simQ (2T+2V)|_{T^\nu} \sim (M_1 + V)|_{T^\nu} \simQ M_1|_{T^\nu} + V|_{T^\nu} \geQ 0
				\]
				it follows that $V|_{T^\nu} \simQ 0$.
			\end{itemize}

			We conclude that
			\begin{itemize}
				\item
				$\Supp T\cap \Supp V =\varnothing$, hence $(*)$ implies that a general $M \in \mathfrak M$ has no vertical part;
				\item
				by Lemma~\ref{lem:adjunction} and Proposition~\ref{prop:char-abel-var}, $T = T^\nu$ is isomorphic to an abelian variety;
				\item
				$(K_{X} + M +\Delta)|_{M} \equiv 0$.
			\end{itemize}
			This completes the proof of the lemma.
		\end{proof}
		
		We proceed with the proof of the proposition.
		We first prove the assertions (i) and (ii) by dividing the argument into two cases based on whether $\deg_{K(S)} (K_X + M_0 + \Delta)$ equals zero or not.
		
		In the following we take a general divisor $M \in \mathfrak M$.
		Remark that to prove that $M_0$ is semi-ample, it suffices to verify that on any irreducible component $T$ of $M$ the restriction $ M|_{T} \equiv 0$, which implies that $|M_0|$ is base-point-free with $\nu(M_0) =1$.
		\medskip

		Case\thinspace (1): $\deg_{K(S)} (K_X + M_0 + \Delta) <0$. In this case $\deg_{K(S)} M_0 =1$, a general $M\in |M_0|$ is reduced and irreducible,  and $(X_{K(S)}, \Delta_{K(S)})$ is klt. Let $ N = -(K_{X} + M +\Delta)$. Then $ N|_{M^\nu} \equiv 0$ by Lemma~\ref{lem:V=0}.
		
		Take an ample divisor $H$ on $S$. For any rational number $\epsilon >0$, the divisor $N+ \epsilon f^*H$ is nef and big, thus $\mathbb Q$-effective.
		Therefore
		$$  K_X+ M +\Delta+ \text{(effective)} \sim_{\mathbb{Q}} \epsilon f^*H.$$
		Thus, by Theorem~\ref{thm:can-bd-g0}, we have $\epsilon H\geQ \frac12 K_{S}$.
		But $K_S$ is $\mathbb Q$-effective by Proposition~\ref{prop:char-abel-var} (1), thus letting $\epsilon\to0$, we obtain $K_S\simQ 0$.
		Therefore, $S$ is an abelian variety by Proposition~\ref{prop:char-abel-var} (3).
		
		Next, we show that $M_0$ is semi-ample by verifying that the restriction  $ M|_{M} \equiv 0$.
		Take a surface $W$ which is the intersection of $\dim X -2$ very general hyperplane sections on $X$ intersecting the base locus of $|M|$ properly. It suffices to show that $\nu(M|_W) =1$. 
		Otherwise, we have $(M |_W)^2>0$, since both $M|_W$ and $N|_W$ are nef but not numerically trivial and $( M |_W)\cdot ( N |_W)=0$ by Lemma \ref{lem:V=0}, Hodge Index Theorem implies that $( N|_W)^2 <0$, which contradicts the fact that $ N|_W$ is nef.
		
		\medskip
		Case\thinspace (2): $\deg_{K(S)} (K_X + M_0 + \Delta) = 0$.
		
		Since $(K_X + M_0 + \Delta)|_{X_{K(S)}}\simQ0$, by Lemma~\ref{lem:relativetrivial} there exists a big open subset $S^{\circ}\subseteq S$ and a pseudo-effective $\mathbb{Q}$-Cartier $\mathbb{Q}$-divisor $D$ on $S$ such that $-(K_X + M_0 + \Delta)|_{X_{S^{\circ}}}\sim_{\mathbb{Q}} f^*D^\circ$, where $D^\circ:=D|_{S^\circ}$.
		By Theorem~\ref{thm:can-bd-g0}, we have $-D^\circ \geQ \frac12 K_{S^{\circ}}$.
		On the other hand, since $S$ is of m.A.d., $K_{S^{\circ}} \geQ 0$.  We conclude that $K_S \simQ 0$ which implies that $S$ is an abelian variety.

		Next, we prove that $M_0$ is semi-ample. With the notation at the beginning $M= T+M'$, by Lemma~\ref{lem:V=0}, $M'|_{T} \equiv 0$, thus we only need to show that
		$T|_{T} \equiv 0$. We divide the argument into the following cases.
		\smallskip
		
		Case\thinspace (2.I): $M' \neq 0$. If $M'= T$ then we are done since $M'|_{T^\nu}\equiv 0$. Assume $M' \neq T$, then $M' \cap T = \emptyset$.  We apply a similar strategy as in Case (1). Take a surface $W$ which is the intersection of $\dim X -2$ very general hyperplane sections on $X$ intersecting the base locus of $M$ properly. Since both $M'|_W$ and $T|_W$ are nef but not numerically trivial and $(M'|_W)\cdot (T|_W)=0$, we can apply Hodge Index Theorem to show that 
		$(T|_W)^2 = (M'|_W)^2=0$ as in  in Case (1). This indicates $T|_T \equiv 0$.
		
		\smallskip
		In the remaining cases, we assume that $M$ is reduced and irreducible.
		
		\smallskip
		Case\thinspace (2.II): $\deg_{K(S)} M_0=1$. In this case, $\deg_{K(S)} \Delta> 0$, thus $\Supp\Delta$ must contain a horizontal component $T_1$. By Lemma~\ref{lem:V=0}, we know that $T_1$ does not intersect a general $M$, but the family of divisors in $|M_0|$ covers $X$, thus $T_1$ is a component of certain $M_1 \in |M_0|$ and $\deg_{K(S)} T_1 =1$. We may write that
		$$M_1 = T_1+ V_1\ \ \mathrm{and}\ \ \Delta = aT_1+ \Delta',\quad 0<a\le 1$$
		where $V_1$ is a vertical divisor and $a$ is the coefficient of $T_1$ in $\Delta$.
		Then by the adjunction formula, we have
		\begin{equation}\label{eq:0ZSI}
			\begin{aligned}
				(K_X + M_1+ \Delta)|_{T^\nu_1} &\simQ (K_X + T_1 + aM_1 + (1-a)V_1 +  \Delta' )|_{T^\nu_1} \\
				&\simQ K_{T^\nu_1} + \Delta_{T^\nu_1}+ aM_1|_{T^\nu_1} + (1-a)V_1|_{T^\nu_1} +  \Delta'|_{T^\nu_1} .
			\end{aligned}
		\end{equation}
		Since $K_{T^\nu_1}, M_1|_{T_1^\nu}, V_1|_{T^\nu_1},\Delta'|_{T_1^\nu} \geQ 0$ and $K_X + M_1+ \Delta$ is anti-nef, we conclude that
		$$(K_X + M_1+ \Delta)|_{T^\nu_1}\equiv0,\quad K_{T^\nu_1} \simQ \Delta_{T^\nu_1} \simQ  M_1|_{T^\nu_1} \simQ (1-a)V_1|_{T^\nu_1} \simQ \Delta'|_{T^\nu_1} \simQ 0.$$
		Since $M_1=T_1+V_1$ is connected by Lemma~\ref{lem:conn}, thus if $a<1$ we obtain $V_1=0$ by $V_1|_{T^\nu_1} \simQ 0$; and if $a=1$, then by $M|_{T^\nu_1} \sim_{\mathbb{Q}} (T_1+V_1)|_{T^\nu_1} \sim_{\mathbb{Q}} 0$ and by the assumption that $T_1|_{T^\nu_1}$ is pseudo-effective, we also have $V_1|_{T^\nu_1} \simQ 0$ and thus $V_1=0$. In turn we conclude that $M_1 = T_1$ and $M|_{T_1^\nu} \simQ 0$.
		Therefore $|M|$ is semi-ample with $\nu(M)=1$.
		\smallskip
		
		Case\thinspace (2.III): $\deg_{K(S)} M_0 =2$. We fall into one of the following two cases:
		\begin{itemize}[left=5em..6em]
			\item[(III-1)]
			The generic fiber $X_{K(S)}$ is not geometrically normal (this can happen only when $p=2$), that is, $X_{K(S)}$ is a non-smooth conic over $K(S)$.
			\item[(III-2)]
			$X_{K(S)}$ is smooth over $K(S)$.
		\end{itemize}
		Fix a general divisor $M_1\in |M_0|$, which is a prime divisor.  By Lemma~\ref{lem:V=0}, $M_1$ is an abelian variety.
		Let us rename it by $S':= M_1$ and consider the base change via the degree two morphism $S'\to S$.
		Note that in each of the above two cases, $X_{K(S')}$ is integral by Proposition~\ref{prop:ga0} or \cite[Lemma~1.3]{Sch10}.
		Let $\nu\colon X_1:=(X_{S'})^{\nu}\to X_{S'}$ be the normalization morphism. Let $f_1\colon X_1 \to S_1$ be the fibration arising from the Stein factorization of $X_1\to S'$. We have the following commutative diagram
		$$\xymatrix{
			&X_1 \ar[d]^{f_1}\ar[r]^<<<<{\nu}\ar@/^7mm/[rr]|{\;\pi\;} &X_{S'}\ar[r]^{\pi'}\ar[d] &X\ar[d]^{f}\\
			&S_1\ar[r]   &S'\ar[r] &S\rlap. \\
		}$$
		
		In case~(III-1), by Proposition~\ref{prop:ga0} (3), $X_{S'}$ is  not normal along
		$\pi'^{-1}M_1$, and $S_1 \to S'$ is a finite purely inseparable morphism of degree two.
		Moreover, if we denote by $T_1$ the prime divisor such that $T_1 = \Supp \pi^*M_1$, then $T_1$ is a birational section over $S_1$ and $\pi^*M_1 = T_1$.
		We may write that
		$$\pi^*K_{X} \sim K_{X_1} + T_1 + V_1,$$
		where $V_1 \geq 0$ (by Proposition~\ref{prop:JiWal-refined-version}) is a vertical divisor over $S_1$. Hence
		$$\pi^*(K_{X} + M + \Delta) \simQ K_{X_1} + T_1 + \pi^*M +  V_1 + \pi^*\Delta .$$
		Applying the adjunction formula on $T_1$, we have 
		$$(K_{X} + M + \Delta) |_{T^\nu_1} \simQ (K_{X_1} + T_1 + \pi^*M + V_1 + \pi^*\Delta)|_{T^\nu_1} \sim_\mathbb Q  K_{T^\nu_1} + \Delta_{T^\nu_1} +  (\pi^*M + V_1 + \pi^*\Delta)|_{T^\nu_1}.$$
		Since $K_{X} + M + \Delta$ is anti-nef and $K_{T^\nu_1}, \Delta_{T^\nu_1},  \pi^*M|_{T^\nu_1},  (V_1 + \pi^*\Delta)|_{T^\nu_1}\ge 0$, we conclude that  $\pi^*M|_{T^\nu_1}=0$. Since $\pi^*M\sim T_1$, we have $T_1|_{T_1^{\nu}} \geQ 0$, which implies that $M|_M \equiv 0$.
		
		In case~(III-2), we have $S_1=S'$ and
		$$\pi^*K_{X} \sim K_{X_1} + V_1$$
		for some effective divisor  $V_1$ vertical over $S_1$.
		
		If $M_1 \to S$ is purely inseparable,  then $\pi^*M_1 = 2T_1$ where $T_1$ is a birational section over $S_1$.
		Applying the adjunction formula, we have
		$$(K_{X} + M + \Delta) |_{T^\nu_1} \simQ (K_{X_1} + T_1 + T_1 + V_1 + \pi^*\Delta)|_{T^\nu_1} \sim_\mathbb Q  K_{T^\nu_1} + \Delta_{T^\nu_1} +  (T_1 + V_1 + \pi^*\Delta)|_{T^\nu_1}.$$
		Since $T_1|_{T^\nu_1} \sim_{\mathbb{Q}} \frac{1}{2}\pi^*M_1|_{T^\nu_1} \geQ 0$, we conclude as before that $\pi^*M|_{T^\nu_1} \sim_{\mathbb{Q}} 0$, which implies that $M|_{M} \sim_{\mathbb{Q}} 0$.
		
		Now consider the case that for general $M \in |M_0|$, $M \to S$ is a separable morphism of degree two. Since $M$ is isomorphic to an abelian variety, the natural morphism $M \to S$ is an \'{e}tale morphism of abelian varieties. Note that
		\begin{itemize}
			\item
			up to isomorphism, there are only finitely many abelian varieties  \'{e}tale over $S$ of degree two;
			\item
			if $M\in |M_0|$ is birationally equivalent to $M_1$ then $K(M_1)\otimes_{K(S)} K(M) \cong K(M)\times K(M)$, hence $\pi^*M$ splits into two distinct components.
		\end{itemize}
		From this we conclude that general $M$ belongs to the same birational equivalent class, and $\pi^*M$ splits into the sum of two divisors $T_1 + T_2$ which varies as $M$ varies.
		We may consider the fibration $X_{S'} \to S'$ and apply the argument of Case~(2.I) to conclude that $M|_{M}\equiv 0$.
		\medskip
		
		We now prove the statement (iii). 
		By (ii), the linear system $|M_0|$ induces a fibration $g\colon X\to C$ onto a smooth curve $C$.
		Since the generic fiber $X_\eta$ of $f\colon X\to S$ has arithmetic genus zero and is dominant over $C$, the curve $C$ is isomorphic to $\mathbb P^1$.
		\medskip
		
		Finally, let us prove the statement (iv). 
		Denote by $G_t$ the fiber of $g\colon X\to \mathbb P^1$ over $t\in\mathbb P^1$. 
		Let $T$ be an $f$-horizontal component of $G_t$. Write that 
		\[
		G_t = aT + G_t' \text{ \ and \ } \Delta = bT + \Delta'
		\]
		with $(G_t' + \Delta')\wedge T = 0$, where $0\le b\le 1$.

		If $b<1$, setting $c = (1-b)/a$, we have
		\[ 
		(K_X + cG_t + \Delta)|_{T^\nu} = (K_X + T + cG_t' + \Delta')|_{T^\nu} \simQ  K_{T^\nu} + \Delta_{T^\nu} + cG_t'|_{T^\nu} + \Delta'|_{T^\nu},
		\]
		which is anti-nef since $G_t|_{T^\nu} \sim 0$.
		Since $\Delta_{T^\nu},cG_t'|_{T^\nu},\Delta'|_{T^\nu},K_{T^\nu} \geQ0$, it follows that $\Delta_{T^\nu}=G_t'|_{T^\nu}=\Delta'|_{T^\nu}=0$ and $K_{T^\nu} \equiv 0$.
		By Lemma~\ref{lem:adjunction} and Proposition~\ref{prop:char-abel-var}, $T$ is normal and isomorphic to an abelian variety. 
		Moreover, as $c>0$, we have $G_t' |_{ T } \simQ 0$, which implies $G_t' = 0$ since $G_t$ is connected. Therefore, $G_t$ is isomorphic to either an abelian variety or a multiple of an abelian variety.
		By the above argument, we have $\Delta|_{T} \equiv 0$; in turn we conclude that $\mathrm{Supp}\, \Delta$ is contained in finitely many closed fibers of $g$.
		Moreover, since $g$ is a fibration onto a curve, the generic fiber of $g$ is geometrically reduced (see \cite[Corollary~2.5]{Sch10}). 
		Therefore, a general fiber $G_t$ of $g$ is reduced and thus is isomorphic to an abelian variety.

		If $b=1$, then by assumption, $T|_{T^\nu}$ is pseudo-effective, and thus so is $\Delta|_{T^\nu}$. Applying the adjunction formula, we have
		\[ 
		(K_X + G_t + \Delta)|_{T^\nu} = (K_X + T + G_t' + \Delta)|_{T^\nu} \simQ  K_{T^\nu} + \Delta_{T^\nu} + G_t'|_{T^\nu} + \Delta|_{T^\nu}.
		\]
		By a similar argument as in the case $b<1$, we can finish the proof.
		\end{proof}

		\section{Canonical bundle formula for separable fibrations}
		Throughout this section, we work over an algebraically closed field $k$ of characteristic $p>0$. We aim to deduce a canonical bundle formula for a separable fibration. We first treat a general case and obtain the following theorem, which can be regarded as an addendum of Witaszek's result (Theorem~\ref{thm:Wit}).
		
		\begin{thm}\label{thm:red-bach}
		Let $f\colon X\to S$ be a fibration of relative dimension one between normal quasi-projective varieties, where
		$X$ is $\mathbb Q$-factorial.  Let $\Delta$ be an effective $\mathbb{Q}$-divisor on $X$. Let $\tau_1\colon S_1 \to S$ be a finite purely inseparable morphism of height one with $S_1$ normal. Assume that
		\begin{enumerate}[\rm(C1)]
			\item $(X_{K(S)}, \Delta_{K(S)})$ is lc;
			\item $K_X+\Delta \sim_{\mathbb{Q}} f^*D$ for some $\mathbb{Q}$-Cartier $\mathbb{Q}$-divisor $D$ on $S$; and
			\item $X_{K(S_1)}$ is reduced but not normal (which happens only when $p=2,3$).
		\end{enumerate}
		Then there exist finite morphisms $\tau\colon \bar{T} \to S$, $\tau'\colon \bar{T}' \to \bar{T}$ and $\tau_1'\colon \bar{T}' \to S_1$ fitting into the following commutative diagram
		$$ \xymatrix@R=3.5ex@C=4ex{ &\bar{T}' \ar[r]^{\tau'} \ar[d]^{\tau'_1}   &\bar{T}\ar[d]^{\tau} \\
			&S_1\ar[r]^{\tau_1}  &S
		}
		$$
		and an effective $\mathbb{Q}$-divisor $E_{\bar{T}'}$ on $\bar{T}'$ such that
		$$(\tau_1\circ\tau_1')^*D \sim_{\mathbb{Q}} a K_{\bar{T}'} + b\tau'^*K_{\bar{T}} + c\tau_1'^*(\tau_1^*K_S - K_{S_1}) + E_{\bar{T}'}$$
		where $a,b, c \geq 0$ are rational numbers relying on the coefficients of $\Delta_{K(S)}$.
		
		The finite morphisms $\tau'$, $\tau_1$ are purely inseparable, and $\tau$, $\tau_1'$ are also purely inseparable if $f$ is separable. Moreover,
		if $(X_{K(S)}, \Delta_{K(S)})$ is klt, then $c\ge\frac{1-\theta}{p(p-1)}>0$ where $\theta$ is the maximum of the coefficients of $\Delta_{K(S)}$.
		\end{thm}
		
		\begin{proof}
		Since the statements involve only codimension one points, we may assume that $S$ and $S_1$ are both regular.
		Here, once $\bar T$ is constructed over an open subset of $S$, we can replace it with the normalization of $S$ in $K(\bar T)$.
		This ensures that $\tau$ is finite. The same applies to $\bar T'$ and $\tau'_1$.

		Let $\nu\colon X'\to X_{S_1}$ be the normalization morphism. Since $X_{K(S_1)}$ is not normal, we can find an $f$-horizontal irreducible component $T'$ on $X'$ of the conductor divisor. Denote by $\pi\colon X'\to X$ the induced morphism, which is purely inseparable of height one. Let $f'\colon X' \to S'$ be the fibration arising from the Stein factorization of $X' \to S_1$. These varieties fit into the following commutative diagram
		\[
		\xymatrix{&T'\ar[r] &X'\ar[d]^{f'}\ar[rd]^{f_1}\ar@/^7mm/[rr]|{\;\pi\;}\ar[r]^{\nu} &X_{S_1} \ar[d]\ar[r] &X\ar[d]^f &T:=\pi(T')\ar[l]\\
			&  &S'\ar[r] &S_1\ar[r]^{\tau_1}   &S  &
		}\]
		
		
		By Proposition~\ref{prop:compds}, we have
		\begin{equation}\label{eq:can-sec5}
			\pi^*(K_X + \Delta) \simQ K_{X'}+ (p-1)(\gamma_1T' + E') + \pi^*\Delta + f_1^*(\tau_1^*K_S - K_{S_1}) + N,
		\end{equation}
		where $E'$ is effective, $T'\wedge E'=0$ and $N$ is exceptional over $S$.
		Since $p_a(X'_{K(S')})=0$, we have $\deg_{K(S')}(p-1)(\gamma_1 T' + E' + \pi^* \Delta) = 2$, thus $\deg_{K(S')}T'=\gamma_1=1$ if $p=3$ and  $\gamma_1\cdot\deg_{K(S')} T'  \le 2$  if $p=2$.
		We may assume $N=0$ since it does not affect our result.
		
		Let $\rho\colon T'^\nu \to T^\nu$ be the morphism arising from the normalization of $T'\to T$.  We may write $\pi^*T =  \gamma_2T'$ where $\gamma_2 = 1$ or $p$.  Write $\Delta= \alpha T+\Delta'$ with $T\wedge \Delta'=0$. 
		By the adjunction formula there exists an effective $\mathbb{Q}$-divisor $\Delta_{T'^\nu}$ on $X'$ such that
		$(K_{X'} + T')|_{T'^\nu}  \sim_{\mathbb{Q}} K_{T'^\nu} + \Delta_{T'^\nu}$.
		By the relation (\ref{eq:can-sec5}) we have
		\begin{equation}\label{eq:can-sec6-1}
			\begin{split}
				\pi^*(K_X+\Delta)|_{T'^\nu}
				&{}\sim_{\mathbb{Q}} (K_{X'} + T')|_{T'^\nu} +  \bigl((p-1)\gamma_1+\alpha  \gamma_2-1\bigr)T'|_{T'^\nu} \\
				& \quad \quad+\bigl((p-1)E'+\pi^*\Delta'\bigr)|_{T'^\nu} + f_1^*(\tau_1^*K_S - K_{S_1})|_{T'^\nu}\\
				&{}\sim_{\mathbb{Q}} K_{T'^\nu} +  \Delta_{T'^\nu} +  \bigl((p-1)\gamma_1+\alpha  \gamma_2-1\bigr)T'|_{T'^\nu} \\
				&\quad \quad+\bigl((p-1)E'+\pi^*\Delta'\bigr)|_{T'^\nu} + f_1^*(\tau_1^*K_S - K_{S_1})|_{T'^\nu}.
			\end{split}
		\end{equation}
		By Lemma~\ref{lem:log-adj}, there exists an effective divisor $B_{T^\nu}$ on $T^\nu$ such that
		\begin{align}\label{eq:can-sec6-2}
			(1-\alpha )\gamma_2T'|_{T'^\nu} \sim (1-\alpha )\rho^*(T|_{T^\nu}) \simQ \rho^*(K_{T^\nu} + B_{T^\nu})|_{T'^\nu} - \pi^*(K_X+\Delta)|_{T'^\nu}.
		\end{align}
		Multiplying the equations (\ref{eq:can-sec6-1}) and (\ref{eq:can-sec6-2}) by $(1-\alpha )\gamma_2$ and $\bigl((p-1)\gamma_1+\alpha  \gamma_2-1\bigr)$ respectively and then summing up, we obtain
		\begin{equation}\label{eq:res-div}
			\begin{split}
				&\bigl((1-\alpha )\gamma_2+ (p-1)\gamma_1 + \alpha \gamma_2-1\bigr)\pi^*(K_X+\Delta)|_{T'^\nu} \\
				&\hbox to 8em{}\sim_{\mathbb{Q}} (1-\alpha )\gamma_2K_{T'^\nu} + \bigl((p-1)\gamma_1+ \alpha  \gamma_2-1\bigr)\rho^*K_{T^\nu} \\
				&\hbox to 10em{} + (1-\alpha )\gamma_2f_1^*(\tau_1^*K_S - K_{S_1})|_{T'^\nu} +E_{T'^\nu},
			\end{split}
		\end{equation}
		where $E_{T'^\nu}:= (1-\alpha )\gamma_2\Delta_{T'^\nu} + (1-\alpha )\gamma_2\bigl((p-1)E'+\pi^*\Delta'\bigr)|_{T'^\nu} 
		+ \bigl((p-1)\gamma_1+\alpha  \gamma_2-1\bigr)\rho^*B_{T^\nu}\ge 0$.
		
		Finally we denote by $\bar{T}, \bar{T}'$ the normalization of $S$ in $T^\nu, T'^\nu$ respectively. By the construction the varieties $\bar{T}, \bar{T}', S_1, S$ fit into the commutative diagram 	$$
		\xymatrix{&T'^\nu\ar[r]^{\sigma} &\bar{T}' \ar[r]^{\tau'} \ar[d]^{\tau'_1}   &\bar{T}\ar[d]^{\tau} &T^\nu\ar[l]\\
			& &S_1\ar[r]^{\tau_1}  &S  &
		}
		$$
		as claimed in the theorem.
		We push down the relation (\ref{eq:res-div}) via $\sigma$ to $\bar{T}'$ and obtain the relation
		\begin{multline*}
			(\tau_1'\circ\tau_1)^*\bigl( \gamma_2 + (p-1)\gamma_1 -1 \bigr)D \\
			\sim_{\mathbb{Q}} (1-\alpha)\gamma_2K_{\bar{T}'} +\bigl((p-1)\gamma_1+\alpha \gamma_2-1\bigr)\tau'^*K_{\bar{T}} + (1-\alpha)\gamma_2\tau_1'^*(\tau_1^*K_S - K_{S_1}) + \sigma_*E_{T'^\nu}.
		\end{multline*}
		The rational numbers $a,b,c$ and the divisor $E_{\bar{T}'}$ are determined by the above equation.
		In particular,  if $\theta$ denotes the maximum of the coefficients of $\Delta_{K(S)}$, then we have
		$$c=\frac{(1-\alpha)\gamma_2}{ \gamma_2+ (p-1)\gamma_1 -1 }\ge \frac{1-\theta}{p(p-1)}.$$
		\medskip
		
		Moreover, if $f$ is a separable fibration, applying Proposition \ref{prop:ga1-geo-reduced-p=2}, we know that both $\tau$ and $\tau_1'$ are purely inseparable by construction.
		\end{proof}

		\begin{thm}\label{thm:sep-cb-formula}
		Let $f\colon X\to S$ be a separable fibration of relative dimension one between normal quasi-projective varieties, where $X$ is $\mathbb Q$-factorial.
		Let $\Delta$ be an effective $\mathbb{Q}$-divisor on $X$.
		Assume that
		\begin{enumerate}[\rm(C1)]
			\item $(X_{K(S)}, \Delta_{K(S)})$ is lc; and
			\item $K_X+\Delta\sim_{\mathbb{Q}} f^*D$ for some $\mathbb{Q}$-Cartier $\mathbb{Q}$-divisor $D$ on $S$.
		\end{enumerate}
		Then there exist finite purely inseparable morphisms $\tau_1\colon \bar{T} \to S$, $\tau_2\colon \bar{T}' \to \bar{T}$, an effective $\mathbb{Q}$-divisor $E_{\bar T'}$ on $\bar T'$ and rational numbers $a,b,c\geq 0$ such that
		$$\tau_2^* \tau_1^*D \sim_{\mathbb{Q}} a K_{\bar{T}'} + b\tau_2^*K_{\bar{T}} + c\tau_2^*\tau_1^*K_S + E_{\bar T'}.$$
		Moreover, if $(X_{K(S)},\Delta_{K(S)})$ is klt, then $c \geq c_0$ for some positive number $c_0$ relying only on the maximal coefficient of prime divisors in $\Delta_{K(S)}$.
		\end{thm}
		\begin{proof}
		If the generic fiber of $f$ is smooth then we may apply \cite[Theorem~3.4]{Wit21}, which tells that there exists a finite purely inseparable morphism $\tau \colon T \to S$ such that
		$$\tau^*D \sim_{\mathbb{Q}} t\tau^*K_S + (1-t)(K_T+\Delta_T)$$
		for some rational number $t\in [0,1]$ and some effective $\mathbb{Q}$-divisor $\Delta_T$ on $T$. Here we remark that when $(X_{K(S)}, \Delta_{K(S)})$ is klt, the argument of \cite[Theorem~3.4]{Wit21} actually shows that $t\ge c_0 > 0$, where $c_0$ relies on the maximal coefficient of prime divisors in $\Delta_{K(S)}$. We may set $\bar{T}' = \bar{T} = T$ to get our assertion as a special case.
		
		Now assume that the generic fiber of $f$ is not smooth. Since $f$ is separable, if we set $\tau_1 = F_S\colon S_1 = S^{\oneoverp} \to S$, then $X_{K(S_1)}$ is integral but not smooth. We may apply Theorem~\ref{thm:red-bach} and obtain the assertion once noticing that $\tau_1^*K_S - K_{S_1} \sim_{\mathbb{Q}} \tau_1^*\bigl(\frac{p-1}{p}K_S\bigr)$.
		\end{proof}

		\section{Canonical bundle formula for inseparable fibrations}
		In this section, we shall treat inseparable fibrations of relative dimension one. We work over an algebraically closed field $k$ of characteristic $p>0$.
		
		\subsection{A special base change}\label{sec:nota-sett}
		Let $X, S$ be normal projective varieties over $k$.
		Let $f\colon X \to S$ be an inseparable fibration of relative dimension one such that the generic fiber $X_{K(S)}$ has arithmetic genus $\le1$.
		In the following, we shall treat the case when $S$ has m.A.d.  Remind that $\Omega_{S^{\oneoverp}/S}^1$ is not necessarily globally generated. We construct a base change $S_1 \to S$ such that $\Omega_{S_1/S}^1$ is generically globally generated as follows.
		\medskip
		
		Assume that $S$ is of m.A.d.\ and denote by $a_S\colon S \to A$ the Albanese morphism of $S$.
		Let $X_1, S_1$ be the normalization of $(X_{A^{\oneoverp}})_{\mathrm{red}}, (S_{A^{\oneoverp}})_{\mathrm{red}}$ respectively. Note that the natural morphism $f_1\colon X_1 \to S_1$ is not necessarily a fibration; we denote by $f'_1\colon X_1 \to S_1'$ the fibration arising from the Stein factorization of $f_1$.  We have the following commutative diagram
		$$\centerline{\xymatrix@C=1.5cm{%
			&X^{\oneoverp}\ar[d]\ar[r]^{\pi'}\ar@/^3pc/[rrr]|{\;F_X\;} &X_1\ar[rd]^{f_1}\ar[d]_{f_1'} \ar@/^1.5pc/[rr]|{\;\pi_1\;}\ar[r]
			&X_{S_1}\ar[r]\ar[d] &X\ar[d]^{f}\\
			&S^{\oneoverp}\ar[rrd]\ar[r]  &S_1'\ar[r]^{\tau'_1}\ar[rd]^{a_{S_1'}}   &S_1\ar[r]^{\tau_1} \ar[d]^{a_{S_1}}  &S\ar[d]^{a_S} \\
			&  &           &A^{\oneoverp} \ar[r]^{F_A}          &A \rlap.
			}}$$
		Note that $a_{S_1}$ is the Albanese morphism of $S_1$ by the universal property of the Albanese morphism $S^{\frac1p} \to A^{\frac1p}$.
		
		We make the following important remark by results of Section~\ref{subsec:fol-bsch}
		\begin{itemize}
		\item
		$X_1= (X_{A^{\oneoverp}})_{\mathrm{red}}^{\nu}$ coincides with $(X_{S_1})_{\mathrm{red}}^{\nu}$, the sheaves
		$$\Omega_{A^{\oneoverp} \to X^{\oneoverp}}\subseteq \Omega_{S_1\to X^{\oneoverp}} \subseteq\Omega_{X_1\to X^{\oneoverp}} \cong \pi'^*\Omega_{X_1/X}^1$$
		coincide with each other over a nonempty open subset of $X^{\oneoverp}$, and the morphisms $X^{\oneoverp} \to X_1$ and $S^{\oneoverp} \to S_1$ are induced by the foliations $(\Omega_{A^{\oneoverp}\to X^{\oneoverp}})^{\perp}$ and $(\Omega_{A^{\oneoverp}\to S^{\oneoverp}})^{\perp}$ respectively.
		\end{itemize}

		\begin{prop}\label{prop:basic-of-X-S-A}
		Let the notation be as above.
		\begin{enumerate}[\rm(1)]
			\item
			Over an open subset of $S_1$,
			$\Omega_{S_1/S}^1$ is globally generated by
			$a_{S_1}^*H^0(A^{\oneoverp}, \Omega_{A^{\oneoverp}/A}^1)$, and the natural map
			$H^0(A^{\oneoverp}, \Omega_{A^{\oneoverp}/A}^1) \to H^0(S_1, \Omega_{S_1/S}^1)$
			is injective.
			\item
			We have $h^0(S_1, \tau_1^*K_S- K_{S_1}) \geq 1$, and if $a_S\colon S \to A$ is inseparable then the strict inequality holds.
			\item
			If $X_{K(S_1)}$ is integral, then there exist an effective divisor $E_1$ and an $f_1$-exceptional divisor $N_1$ on $X_1$
			such that
			$$\pi_1^*K_X \sim_{\mathbb Q} K_{X_1} + E_1 + f_1^*(\tau_1^*K_S- K_{S_1}) + N_1.$$
			\item
			If $X_{K(S_1)}$ is non-reduced, then the movable part of the linear system $\lvert\det (\Omega_{X_1/X}^1)\rvert$ has nontrivial horizontal components over $S_1$.
			\item
			If $X_{K(S_1)}$ is normal, then $f_1\colon X_1 \to S_1$ is a fibration. We may do base change $S_2:=(S_1\times_{A^{\oneoverp}} A^{\oneoverp[^2]})_{\mathrm{red}}^{\nu} \to S_1$ as above. Repeating this process we can obtain a number $n$ such that $X_{K(S_{n-1})}$ is normal, but $X_{K(S_n)}$ is not normal.
		\end{enumerate}
		\end{prop}
		\begin{proof}
		(1) First we note that, the natural homomorphism $a_{S_1}^* \Omega^1_{A^{\oneoverp}/A} \to \Omega^1_{S_1/S}$ is generically surjective.
		Since $\Omega^1_{A^{\oneoverp}/A} \cong \Omega^1_{A^{\oneoverp}/k^{1/p}}$ is a trivial vector bundle, over an open subset of $S_1$, $\Omega_{S_1/S}^1$ is globally generated by
		$a_{S_1}^*H^0(A^{\oneoverp}$, $\Omega_{A^{\oneoverp}/A}^1)$.
		For the second assertion, since the morphism $S^{\oneoverp} \to A^{\oneoverp}$ is the Albanese morphisms of $S^{\oneoverp}$, 
		by \cite[Théorème~4]{Serre58}, the natural map
		$H^0(A^{\oneoverp}, \Omega_{A^{\oneoverp}}^1) \cong H^0(A^{\oneoverp}, \Omega_{A^{\oneoverp}/A}^1) \to H^0(S^{\oneoverp}, \Omega_{S^{\oneoverp}/S}^1) \cong H^0(S^{\oneoverp}, \Omega_{S^{\oneoverp}}^1)$ is injective.
		However, this map factors through $H^0(A^{\oneoverp}, \Omega_{A^{\oneoverp}/A}^1) \to H^0(S_1, \Omega_{S_1/S}^1)$, which is injective too.
		
		(2) 
		By the assertion~(1), we have $\det\Omega_{S_1/S}^1 \succeq 0$.  Then by $\tau_1^*K_S= K_{S_1} +(p-1) \det\Omega_{S_1/S}^1$, we have $h^0(S_1, \tau_1^*K_S- K_{S_1}) \geq 1$.
		
		Next, assume $a_S$ is inseparable.
		Let $B$ be the image of $S$ in $A$.
		Since $a_S$ is inseparable,  $S \otimes_{K(B)} K(B)^{1/p}$ is non-reduced.
		As $S_1 = (S \times_B B^{\frac1p})_{\rm red}^\nu$, by the assertion (ii) in Section~\ref{subsec:fol-bsch}, we have $\mathop{\rm rank} \Omega^1_{S_1/S} < \mathop{\rm rank} \Omega^1_{B^{\oneoverp}/B} = \dim S$.
		Now, using the assertion (1), we have
		$$h^0(S_1, (\Omega_{S_1/S}^1)^{\vee\vee}) \geq h^0(A^{\oneoverp}, \Omega_{A^{\oneoverp}/A}^1)= \dim A\geq \dim S >  \mathrm{rank}~\Omega_{S_1/S}^1.$$
		Then by Lemma~\ref{lem:kod-gg}, we have $h^0(S_1, \det\Omega_{S_1/S}^1) \ge 2$, which implies that $h^0(S_1, \tau_1^*K_S- K_{S_1}) \geq 2$.
		
		The assertion~(3) is a direct consequence of Proposition~\ref{prop:compds}.

		(4) Note that $X_1$ coincides with $(X_{S_1})_{\mathrm{red}}^{\nu}$.
		By assertion~(1), over an open subset $U$ of $S_1$, $\Omega_{S_1/S}^1$ is globally generated by $\Gamma:=a_{S_1}^*H^0(A^{\oneoverp}, \Omega_{A^{\oneoverp}/A}^1)\subseteq H^0(S_1,\Omega_{S_1/S}^1)$.
		Shrinking $U$, we may assume $\Omega_{U/S}^1$ is locally free.
		Then we can apply Proposition~\ref{prop:JiWal-refined-version} to conclude the assertiont.
		
		(5) If $X_{K(S_1)}$ is normal, then the inclusion $\mathcal{O}_{S_1} = \tau_1^*f_{*} \mathcal{O}_{X} \subseteq f_{1*} \mathcal{O}_{X_1}$ is an isomorphism over the generic point of $S_1$. In turn, since $S_1$ is normal, we conclude that $f_1\colon X_1 \to S_1$ is a fibration.
		
		Denote by $B$ the Albanese image $a_X(X)$. If $K(S)/K(B)$ is inseparable, then $[K(S_1):K(B^{\oneoverp})] < [K(S):K(B)]$. Finally we may attain some $m$ such that $S_{m} \to B^{\oneoverp[^m]}$ is separable, which implies that $S_{m+1} \cong S_m^{\oneoverp}$.
		Since $X_{K(S)}$ is not geometrically normal, there must exist some $n$ such that $X_{K(S_n)}$ is not normal.
		\end{proof}

		\subsection{Canonical bundle formula for inseparable fibrations} Our first result for inseparable fibrations is the following, under the condition that the Albanese morphism $a_S\colon S \to A$ is separable.
		
		\begin{thm}\label{thm:base-sep-A}
		Let $X$ be a normal $\mathbb{Q}$-factorial projective variety and $S$ be a normal projective variety.
		Let $\Delta$ be an effective $\mathbb{Q}$-divisor on $X$.
		Let $f\colon X \to S$ be an inseparable fibration of relative dimension one. Assume that
		\begin{enumerate}[\rm(C1)]
			\item $(X_{K(S)}, \Delta_{K(S)})$ is lc;
			\item there exists a big open subset $S^{\circ}$ contained in the regular locus $S^{\rm reg}$ of $S$ and a $\mathbb{Q}$-divisor $D^{\circ}$ on $S^{\circ}$ such that
			$(K_X+\Delta)|_{f^{-1}(S^{\circ})}  \sim_\mathbb Q f^*D^{\circ}$;
			\item $S$ is of m.A.d., and the Albanese morphism $a_S\colon S \to A$ is separable.
		\end{enumerate}
		Then $D \geQ \frac{1}{2p}K_S$, where $D:=\overline{D^{\circ}}$ is the closure divisor of $D^{\circ}$ in $S$. In particular, $\kappa(S,D) \geq \kappa(S)$.
		\end{thm}
		\begin{proof}
		To prove the assertion, we may restrict ourselves on $S^{\circ}$. So in the following, we assume $S=S^{\circ}$. By the assumption~(C3) we have $S_1 = S^{\oneoverp}$. Since the fibration $X^{\oneoverp} \to S^{\oneoverp}$ factors through $f_1\colon X_1 \to S_1$, the morphism $f_1$ is also a fibration.
		By Proposition~\ref{prop:curve-nonreduced}, the generic fiber of $X_1 \to S_1$ has arithmetic genus zero, and
		$$\pi_1^*K_X \sim_{\mathbb Q} K_{X_1} + (p-1)\det (\Omega_{X_1/X}^1).$$
		By Proposition \ref{prop:basic-of-X-S-A} (4), the movable part of the linear system $\lvert \det (\Omega_{X_1/X}^1) \rvert$ has nontrivial horizontal components.
		We may write that
		$$f_1^*\tau_1^*D \sim_{\mathbb Q}\pi_1^*(K_X + \Delta) \sim_{\mathbb Q} K_{X_1} + (p-1)\abs{\det (\Omega_{X_1/X}^1)} + \pi_1^*\Delta.$$
		Applying Theorem~\ref{thm:can-bd-g0}, we have $\tau_1^*D\geQ \frac{1}{2}K_{S_1}$. In turn, by $\tau_1^*K_S \sim_{\mathbb{Q}} pK_{S_1}$ we conclude that $D \geQ\frac{1}{2p}K_{S}$.
		\end{proof}
		
		The second theorem treats the case when the Albanese morphism $a_S\colon S \to A$ is inseparable.
		
		\begin{thm}\label{thm:insep-A}
		Let the notation be as in Theorem~\ref{thm:base-sep-A}. Assume {\rm(C1,\thinspace C2)} and the following condition
		\begin{enumerate}
			\item[\emph{(C3$'$)}] $S$ is of m.A.d., and the Albanese morphism $a_S\colon S \to A$ is inseparable.
		\end{enumerate}
		Then $\kappa(S,D)\ge0$ where $D:=\overline{D^{\circ}}$. More precisely, letting $X_1$, $S_1$, $f_1$ be as in Section~\ref{sec:nota-sett}, we have:
		\begin{enumerate}[\rm(a)]
			\item if $X_{K(S_1)}$ is integral, then $\tau_1^*D \geQ \tau_1^*K_S - K_{S_1}$, hence $\kappa(S, D) \geq 1$;
			\item if $X_{K(S_1)}$ is non-reduced, then $\tau_1'^*\tau_1^*D \geQ \frac{1}{2}K_{S'_1}$, and if moreover $\kappa(S,D)= 0$ and  $S_1'$ admits a resolution of singularities, then $S_1'\to A^{\oneoverp}$ is birational, and thus $S_1 = S_1'$.
		\end{enumerate}
		In particular, if $\dim S=2$ then $\kappa(S,D)\ge 1$.
		\end{thm}
		\begin{proof}
		We may also restrict us on $S^{\circ}$ and assume $S=S^{\circ}$.
		
		If $X_{K(S_1)}$ is normal, by Proposition~\ref{prop:basic-of-X-S-A}~(3) there exist divisors $E_1\ge0$ and $N_1$ (exceptional over $S_1$) fitting into the following equation
		\begin{equation}
			\label{eq:7.5-1}
			\pi_1^*(K_{X} + \Delta) \sim_{\mathbb Q} K_{X_1} + E_1 + \pi_1^*\Delta + f_1^*(\tau_1^*K_{S}- K_{S_1}) + N_1.
		\end{equation}
		Let $\Delta_1 = E_1 +\pi_1^*\Delta$ and $D_1= \tau_1^*D + K_{S_1} - \tau_1^*K_{S}$, then \[
		K_{X_1} + \Delta_1+N_1\sim_\mathbb Q f_1^*D_1.
		\]
		Repeating this process, by Proposition~\ref{prop:basic-of-X-S-A} (5), we obtain a minimal $n$ such that $(X_{n-1})_{K(S_n)}$ is not normal.
		We have the following commutative diagram: \[
		\xymatrix{
			&X_n\ar[r]^{\pi_{n}}\ar[d]^{f_n}\ar[dl]_{f_n'} & X_{n-1}\ar[d]^{f_{n-1}}\\
			S_n'\ar[r]^{\tau'_n}&S_n\ar[r]^{\tau_n}&S_{n-1}
		}
		\]
		where $f'_n$ is a fibration.
		
		Note that $(\Delta_{n-1})_{K(S_{n-1})} = \Delta_{K(S_{n-1})}$.
		We claim that $(X_{K(S_{n-1})},\Delta_{K(S_{n-1})})$ is lc. By construction, for $1\leq i\leq n-1$, each $X_{K(S_{i})}= X_{K(S)}\otimes_{K(S)} K(S_{i})$ is not geometrically reduced. It is trivial if $X_{K(S_{n-1})}$ has arithmetic genus one because then $\Delta_{K(S_{n-1})} =0$. We only need to consider the case $p=2$ and $X_{K(S_{n-1})}$ is a non-smooth conic over $K(S_{n-1})$, on which each prime divisor has degree $\geq 2$. Since $\deg_{K(S_{n-1})}K_{X_{K(S_{n-1})}} = -2$ and $K_{X_{K(S_{n-1})}} + \Delta_{K(S_{n-1})}\equiv 0$, we see that if $(X_{K(S_{n-1})},\Delta_{K(S_{n-1})})$ is not klt, then it is lc and $\Delta_{K(S_{n-1})}$ is a prime divisor of degree two.
		
		\medskip
		Case\thinspace (1): $(X_{n-1})_{K(S_n)}$ is reduced but not normal. By Proposition~\ref{prop:basic-of-X-S-A}~(3)  we have
		\begin{equation}
			\label{eq:7.5-3}
			K_{X_n} + E_n + \pi_n^*\Delta_{n-1} +N_n \sim_\mathbb Q
			f'^*_n\tau'^*_n(\tau_n^* D_{n-1} + K_{S_n} - \tau_n^* K_{S_{n-1}}) .
		\end{equation}
		If necessary replacing $S$ with a big open subset, we may assume $N_n=0$.
		Applying Theorem~\ref{thm:red-bach} to the fibration $f_{n-1}\colon X_{n-1} \to S_{n-1}$ and the base change $\tau_n\colon S_n\to S_{n-1}$, we see that there exist finite morphisms
		$\tau'\colon \bar{T} \to S_n$, $\tau''\colon \bar{T}' \to \bar{T}$, an effective $\mathbb{Q}$-divisor $E_{\bar T'}$ and rational numbers $a,b,c\geq 0$ such that
		\begin{equation}
			\label{eq:7.5-4}
			(\tau_n^* D_{n-1} + K_{S_n} - \tau_n^* K_{S_{n-1}})|_{\bar{T}'} \sim_{\mathbb{Q}} a K_{\bar{T}'} + b\tau''^*K_{\bar{T}} + c\tau''^*\tau'^*K_{S_n} + E_{\bar T'}.
		\end{equation}
		Since $K_{\bar{T}'},K_{\bar{T}},K_{S_n}\geQ 0$, we see that
		$$\bigl(D_n:=\tau_n^* D_{n-1} - (  \tau_n^* K_{S_{n-1}}-K_{S_n} )\bigr)\big|_{\bar{T}'} \geQ 0.$$
		By Covering Theorem~\ref{thm:covering} we have
		$$\kappa(S_{n-1},D_{n-1})= \kappa(S_n,\tau_n^*D_{n-1})\ge \kappa(S_n,\tau_n^*K_{S_{n-1}}-K_{S_n}).$$
		Remark that
		\begin{itemize}
			\item
			if $S_{n-1} \to A^{\oneoverp[^{n-1}]}$ is separable, then $S_n = S_{n-1}^{\oneoverp}$, therefore $\kappa(S_n,\tau_n^*K_{S_{n-1}}-K_{S_n}) = \kappa(S_{n-1},K_{S_{n-1}}) \geq 0$;
			\item
			otherwise by Proposition~\ref{prop:basic-of-X-S-A} (2), we have $\kappa(S_n,\tau_n^*K_{S_{n-1}}-K_{S_n}) \geq 1$.
		\end{itemize}
		In particular each $D_i  \geQ 0$, $i=1, \ldots, n-1$,  and inductively we obtain that
		$$\tau_1^*D = D_1 + (\tau_1^*K_{S}-K_{S_1}) \geQ \tau_1^*K_{S}-K_{S_1}.$$
		Then by Covering Theorem and the assumption that $a_S$ is inseparable, we have $\kappa(S,D) \geq \kappa(S_1,\tau_1^* K_S - K_{S_1}) \ge1$.
		
		\medskip
		Case\thinspace(2): $(X_{n-1})_{K(S_n)}$ is non-reduced. In this case $\bigl\lvert\det (\Omega_{X_{n}/X_{n-1}}^1)\bigr\rvert$ has nontrivial horizontal movable part by Proposition \ref{prop:basic-of-X-S-A} (4). We have
		$$f_n^*\tau_{n}^*D_{n-1} \sim_{\mathbb Q}\pi_n^*(K_{X_{n-1}} + \Delta_{n-1}) \sim_{\mathbb Q} K_{X_{n}} + (p-1)\det (\Omega_{X_{n}/X_{n-1}}^1) + \pi_n^*\Delta_{n-1}.$$
		Applying Theorem~\ref{thm:can-bd-g0} to the fibration $f_n'\colon X_n \to S_n'$, we see that
		$$\tau_n'^* \tau_n^*D_{n-1} \geQ \frac{1}{2}  K_{S_n'},$$
		thus
		$\kappa(S_n', \tau_n'^*\tau_n^*D_{n-1}) \geq \kappa(S_n', K_{S_n'}) \geq 0$.
		Therefore, $D_{n-1} \geQ 0$.
		
		If $n\geq 2$, using $D_i = \tau_i^*D_{i-1} - (\tau_i^*K_{S_{i-1}}-K_{S_i})$  and $\tau_i^*K_{S_{i-1}}-K_{S_i} \geQ 0$ ($1<i\leq n$)  inductively we prove that $D_1 \geQ 0$. It follows that
		$$\tau_1^*D = D_1 + (\tau_1^*K_{S}-K_{S_1}) \geQ \tau_1^*K_{S} - K_{S_1},$$
		thus $\kappa(S,D) \geq 1$.
		
		If $n=1$, since $\tau_1'^*\tau_{1}^*D\geQ \frac{1}{2}  K_{S_1'}$, applying Covering Theorem~\ref{thm:covering} we obtain
		$$\kappa(S,D) = \kappa(S_1', \tau_1'^*\tau_1^*D) \geq \kappa(S_1', K_{S_1'}) \geq 0.$$
		If moreover $\kappa(S,D)=0$, then $\kappa(S_1', K_{S_1'}) = 0$. Since by assumption $S_1'$ admits a resolution of singularities, we can apply Proposition~\ref{prop:char-abel-var} to show that the Albanese morphism $a_{S_1'}\colon S_1' \to A^{\oneoverp}$ is birational.  In turn, we see that $S_1 = S_1'$.

		Finally, if $\dim S = 2$, then $\deg \tau_1\le p$, thus $X_{K(S_1)}$ is reduced (e.g., \cite[Lemma~1.3]{Sch10}). We fall into case~(1), and it follows that $\kappa(S,D)\ge 1$.
		\end{proof}

		\section{Albanese morphism of varieties with nef anticanonical divisor}\label{sec:alb-morph}
		In this section, we apply the canonical bundle formulas obtained in the previous sections to investigate varieties with nef anticanonical divisor. We
		work over an algebraically closed field $k$ of characteristic $p>0$.
		
		\begin{thm}\label{thm:S-abelian}%
		Let $X$ be a normal projective $\mathbb Q$-factorial variety and $\Delta$ an effective $\mathbb{Q}$-divisor on $X$ such that
		$-(K_X+\Delta)$ is nef.
		Let $X \buildrel f\over\to S \buildrel a_S\over\to A_X$ be the Stein factorization of the Albanese morphism of $X$.
		Suppose that the fibration $f$ has relative dimension one and that $(X_{K(S)},\Delta_{K(S)})$ is klt.
		Then $S$ is an abelian variety.
		\end{thm}
		
		\subsection{A preliminary lemma}
		\begin{lem}\label{lem:fol-fibr}
		Let $X$ be a normal projective variety equipped with two fibrations:
		\[
		\xymatrix@R4ex@C4ex{ X\ar[r]^g \ar[d]_f & \mathbb P^1 \\  A}
		\]
		where $A$ and a general fiber $G_t$ of $g$ are abelian varieties for which the induced map $G_t\to A$ is finite and dominant.
		Let $\mathcal{F}$ be a foliation on $X$ and assume that $\det(\mathcal F)|_{G_t} \simQ 0$ for a general $t\in \mathbb P^1$.
		Then the ``pushing down'' foliation $\mathcal G$ of $\mathcal F$ along $f$ (see \S\ref{sec:pushing-down}) is generated by a subspace of $H^0(A,\mathcal T_A)$.
		In particularly, $A/\mathcal G$ is an abelian variety.
		\end{lem}
		\begin{proof}
		Recall from Lemma~\ref{push-foliation} that, under the natural $\mathcal{O}_X$-linear homomorphism $\eta\colon \mathcal F \subseteq \mathcal T_X \to f^*\mathcal T_A$,
		the foliation $\mathcal{G}$ is the minimal foliation on $A$ such that $\eta(\mathcal{F}) \subseteq f^*\mathcal G$ holds generically.
		
		\smallskip
		Claim. {\it For a general $t\in \mathbb P^1$, there exists a big open subset $G_{t}^\circ$ of $G_t$ such that the image of $\eta_t\colon \mathcal F|_{G_t^\circ}\to (f|_{G_t^\circ})^*\mathcal T_A$ is a free sheaf.}
		
		\begin{proof}[Proof of the claim]
			Let $G_t$ be a general fiber. Since both $X$ and $G_t$ are normal, there exists a big open subset $G_t^\circ \subset G_t$ such that $X$ is regular along $G^\circ_t$, and that both $\mathcal{F}|_{G^\circ_t}$ and $(\mathcal T_X/\mathcal F)|_{G_t^\circ}$ are locally free.
			Note that the normal bundle $\mathcal{N}_{G_t/X}|_{G^\circ_t} \cong \mathcal{O}_{G^\circ_t}$.
			By assumption $\mathcal T_{G_t^\circ} \cong \bigoplus^n\mathcal O_{G_t^\circ}$ where $n := \dim A$.
			Then we have the following commutative diagram of $\mathcal{O}_{G^\circ_t}$-linear homomorphisms
			\[\xymatrix@R=2.5ex@C=3ex@M=6pt{
				&\bigoplus^{n}\mathcal{O}_{G^\circ_t}\ar@{}[d]|{\rotatebox{-90}{$\cong$}}&&\mathcal{O}_{G^\circ_t}\ar@{}[d]|{\rotatebox{-90}{$\cong$}}\\
				\noalign{\vskip-1ex}
				0\ar[r] & \mathcal T_{G^\circ_t}  \ar[r] & \mathcal T_X|_{G^\circ_t} \ar[r] & \mathcal{N}_{G_t/X}|_{G^\circ_t} \ar[r] & 0\\
				0\ar[r]&\mathcal{F}|_{G^\circ_t}\cap \mathcal T_{G^\circ_t}\ar@{_(->}[u]\ar[r]&\mathcal{F}|_{G^\circ_t}\ar@{_(->}[u]\ar[ur]_{\theta} \ar[r]&\mathrm{Im}\,\theta\ar[r]\ar@{_(->}[u]&0
			}\]
			where the two horizontal sequences are exact.
			We have $\det (\mathcal{F}|_{G^\circ_t}\cap \mathcal T_{G^\circ_t}) \leL 0$ and $\det(\im\theta)\leL 0$ since they are subsheaves of free sheaves.
			Therefore, the assumption $\det(\mathcal{F}|_{G^\circ_t}) \simQ 0$ implies that $\det (\mathcal{F}|_{G^\circ_t}\cap \mathcal T_{G^\circ_t}) \sim 0$ and $\det(\im\theta) \sim 0$.
			By shrinking $G_t^\circ$, Lemma~\ref{trivibdl} implies that both $\mathcal{F}|_{G^\circ_t}\cap \mathcal T_{G^\circ_t}$ and $\mathrm{Im}\,\theta$ are free sheaves.
			That is to say, $\mathcal F|_{G_t^\circ}$ is an extension of free sheaves.
			Denote by $\mathcal K$ the kernel of $\eta_t$. Then $\mathcal K$ is a subsheaf of a sheaf which is an extension of free sheaves, thus $\det \mathcal K\leL 0$.
			Since $(\mathcal F|_{G_t^\circ})/\mathcal K \cong \im \eta_t \subseteq \bigoplus^{n}\mathcal{O}_{G^\circ_t}$, we have  $\det ((\mathcal F|_{G_t^\circ})/\mathcal K) \leL 0$. Therefore by $\det(\mathcal{F}|_{G^\circ_t}) \equiv 0$, we conclude that $\det \mathcal K \sim \det(\im\eta_t) \sim 0$.
			Applying Lemma~\ref{trivibdl} again, if necessary replacing $G_t^{\circ}$ with a big open subset, we see that $\im\eta_t$ is a free sheaf.
		\end{proof}
		
		\medskip
		
		There exists a non-empty open subset $V\subseteq \mathbb P^1$ such that the claimed assertion holds for all $t\in V$. Fix  $t\in V$, there exist $\alpha_{t, 1}, \ldots, \alpha_{t, k_t}\in H^0(G_t^{\circ}, \im\eta_t)$ such that $\im\eta_t =\bigoplus_{i=1}^{k_t}\mathcal{O}_{G^\circ_t}\cdot\alpha_{t,i}$.
		Since $G_t^{\circ}$ is a big open subset of $G_t$, we have a natural isomorphism $H^0(A,\mathcal T_A = \mathcal{O}_A^n) \cong H^0(G_t^\circ, f_t^*\mathcal T_A|_{G_t^\circ})$, which is induced by the pullback of the map $f_t:=f|_{G_t^\circ}\colon G_t^\circ \to A$. Let $\beta_{t,i} \in H^0(A,\mathcal T_A)$ be the element corresponding to $\alpha_{t, i}$. Let $\mathcal G_t$ be the subsheaf of $\mathcal T_A$ generated by $\{\beta_{t,i}\}$.
		
		Set $\mathcal{G}' := \sum_{t\in V} \mathcal{G}_t$. It follows that $\mathcal{G}'$ is globally generated by $\Lambda'=H^0(A, \mathcal{G}') \subseteq H^0(A,\mathcal T_A)$. Here we remind that $\Lambda'$ is not necessarily a $p$-Lie subalgebra of $H^0(A,\mathcal T_A)$, because it is not necessarily $p$-closed, in other words, $\mathcal{G}'$ is not necessarily a foliation. Let $\bar{\Lambda}$ be the $p$-Lie subalgebra of $H^0(A,\mathcal T_A)$ generated by $\Lambda'$, which corresponds to a smooth foliation $\overline{\mathcal{G}}$ on $A$. Note that $\overline{\mathcal{G}}$ is the minimal foliation containing $\mathcal{G}'$.

		To finish the proof, we are left to verify that $\overline{\mathcal{G}} = \mathcal{G}$.
		On one hand,  for  $t\in V$, by the construction, $\mathcal G_t$ is the smallest the subsheaf of $\mathcal T_A$ such that $\im\eta_t \subseteq f^*\mathcal G_t|_{G_t^\circ} $, thus $\mathcal G_t \subseteq \mathcal{G}$,
		we conclude that $\mathcal{G}' \subseteq \mathcal{G}$, which implies that $\overline{\mathcal{G}} \subseteq \mathcal{G}$.
		On the other hand, over a nonempty open subset $U \subseteq g^{-1} V$ we have $\eta(\mathcal{F})|_U\subseteq f^*\overline{\mathcal{G}}|_U$,
		but $\mathcal{G}$ is the minimal foliation in this sense, therefore $\mathcal{G} \subseteq \overline{\mathcal{G}}$. To summarize, we obtain that $\overline{\mathcal{G}} = \mathcal{G}$ and finish the proof.
		\end{proof}
		\begin{lem}\label{trivibdl}
		Let $X$ be a regular variety satisfying $H^0(X,\mathcal O_X) = k$, and let $\mathcal{F}$ be a subsheaf of $\bigoplus^{n}\mathcal{O}_X$.
		Then $\det \mathcal{F} \preceq 0$.
		
		Moreover, if $\det \mathcal{F} \sim 0$, then $\mathcal F^{\vee\vee} \cong \bigoplus^r \mathcal O_X$ for some integer $r$.
		\end{lem}
		\begin{proof}
		By replacing $\mathcal F$ with $\mathcal F^{\vee\vee}$, we may assume that $\mathcal F$ is reflexive. 
		We may also assume that $\rank \mathcal F \ge 1$.
		If $n=1$, then $\mathcal F \hookrightarrow \mathcal O_X$ is a subsheaf of rank one and locally free in codimension one.
		Thus $\mathcal F$ defines an effective divisor $D$ such that $\mathcal F\cong \mathcal O_X(-D)$.
		So the statements hold for $n=1$.
		We then prove the lemma by induction on $n$.
		Assume the assertion holds for $n\leq l$ where $l>0$. We prove it for $n=l+1$.
		Write $\mathcal W:=\bigoplus^{l}\mathcal{O}_X \subset \bigoplus^{l+1}\mathcal{O}_X$ for the sum of the first $l$ summands.
		We may assume that $\mathcal{F} \not\subseteq \mathcal W$.
		We have the following commutative diagram
		\[
		\xymatrix@R=3.5ex@C=3ex@M=6pt{
			0\ar[r]&  \mathcal{F}\cap \mathcal W \ar[r]\ar@{^(->}[d] & \mathcal{F}\ar[r] \ar@{^(->}[d]& \mathcal{F}\smash{/(\mathcal F\cap \mathcal W)} \ar[r] \ar@{^(->}[d]&0\\
			0\ar[r]& \mathcal W\ar[r]\ar@{-->}@/^6mm/[u]^{\theta}&\smash{\bigoplus^{l+1}}\mathcal{O}_X\ar[r]\ar@/_5mm/[l]_{\pi}&\mathcal{O}_X\ar[r]&0,
		}
		\]
		where $\pi$ is the projection onto $\mathcal W$.
		We remark that the third column $\mathcal F/(\mathcal F\cap \mathcal W)\to \mathcal O_X$ is injective by Snake Lemma; hence $\mathcal F/(\mathcal F\cap \mathcal W)$ is torsion free.
		Thus, $\rank \mathcal F \cap \mathcal W < \rank \mathcal F$.
		Now, by induction hypothesis, $\det (\mathcal{F}\cap \mathcal W)\preceq 0$ and $\det(\mathcal F/(\mathcal F\cap \mathcal W))\preceq 0$.
		So we have $\det \mathcal F \preceq 0$.
		
		Moreover, if $\det \mathcal{F} \sim 0$, we conclude that
		$\det (\mathcal{F}\cap \mathcal W)\sim \det ( \mathcal F/(\mathcal F\cap \mathcal W))\sim 0$.
		By the induction hypothesis we have $\mathcal F/(\mathcal F\cap \mathcal W)\cong \mathcal{O}_X$ and $(\mathcal{F}\cap \mathcal W) \cong \bigoplus^{r}\mathcal{O}_X$ for some $r$ (here we may shrink $X$ so that these sheaves are reflexive).
		Regard $H^0(X,\mathcal{F}\cap \mathcal W)\cong k^{r}$ as a $k$-linear subspace of $H^0(X,\mathcal W)\cong k^{l}$. Take a splitting $\theta_{k}\colon H^0(X,\mathcal W)\to H^0(X,\mathcal{F}\cap \mathcal W)$, which determines uniquely a splitting $\theta\colon \mathcal W\to \mathcal{F}\cap \mathcal W$ of $\mathcal{F}\cap \mathcal W \hookrightarrow \mathcal W$.
		Thus the composition of homomorphisms $\mathcal{F}\hookrightarrow \bigoplus^{k+1}\mathcal{O}_X\stackrel{\pi}{\to} \mathcal W \stackrel{\theta}{\to} \mathcal{F}\cap \mathcal W$ splits $\mathcal{F}\cap \mathcal W\to \mathcal{F}$.
		As a consequence $\mathcal{F} \cong (\mathcal{F}\cap \mathcal W)\oplus (\mathcal{F}+\mathcal W) /\mathcal W \cong \bigoplus^{r}\mathcal{O}_X \oplus \mathcal{O}_X$, which completes the proof.
		\end{proof}

		\subsection{Proof of Theorem~\ref{thm:S-abelian}}
		We use the notation from Section~\ref{sec:nota-sett} and recall the following commutative diagram
		$$\xymatrix@C=1.5cm{ X_1\ar[rd]^{f_1}\ar[d]_{f_1'} \ar@/^1.5pc/[rr]|{\,\pi\,}\ar[r] &X_{S_1}\ar[r]\ar[d] &X\ar[d]^{f}\\
		S_1'\ar[r]^{\tau'_1}\ar[rd]_{a_{S_1'}}   &S_1\ar[r]^{\tau_1} \ar[d]^{a_{S_1}}  &S\ar[d]^{a_S} \\
		&A^{\oneoverp} \ar[r]^{F_A}          &A\rlap{.}
		}$$
		
		First, in the following lemma, we deal with the cases where one of $f$ and $a_S$ is separable.
		\begin{lem}\label{lem:S-abelian-partI}
		Under the assumptions of Theorem~\ref{thm:S-abelian}, $S$ is an abelian variety unless the following conditions hold simultaneously
		\begin{enumerate}[\rm(i)]
			\item both the morphisms $f$  and $a_S$ are inseparable;
			\item $X_{K(S_1)}$ is non-reduced; and
			\item $a_{S_1}\colon S_1 \to A^{\frac1p}$ is an isomorphism.
		\end{enumerate}
		\end{lem}
		\begin{proof}
		We treat the following two cases separately:
		\begin{itemize}
			\item[(1)] $-(K_X+\Delta)|_{X_{K(S)}}$ is ample,
			\item[(2)] $-(K_X+\Delta)|_{X_{K(S)}}\sim_\mathbb Q 0$.
		\end{itemize}
		
		Case\thinspace (1): $-(K_X+\Delta)|_{X_{K(S)}}$ is ample.
		Let $H$ be an ample divisor on $S$ and $0<\epsilon\ll 1$ a rational number.
		Then $-(K_X+\Delta)+\epsilon f^*H$ is nef and big.
		By Lemma~\ref{lem:nefbigdiv}, $-(K_X+\Delta)+\epsilon f^*H\sim_{\mathbb{Q}} \Delta_{\epsilon}$ for some effective $\mathbb{Q}$-divisor  $\Delta_{\epsilon}$ with coefficients small enough, such that $(X_{K(S)},(\Delta+\Delta_{\epsilon})_{K(S)})$ is klt and the maximal value of the coefficients of $(\Delta+\Delta_{\epsilon})_{K(S)}$ is smaller than some number $\theta <1$ (independent of $\epsilon$). By construction
		we have
		$$K_X+\Delta+\Delta_{\epsilon}\sim_{\mathbb{Q}} \epsilon f^*H.$$

		Case\thinspace (1.1):  The fibration $f$ is separable. We can apply Theorem~\ref{thm:sep-cb-formula} to obtain finite purely inseparable morphisms $\tau_1\colon \bar{T} \to S$, $\tau_2\colon \bar{T}' \to \bar{T}$ and effective $\mathbb{Q}$-divisor $E_{\bar T'}$ on $\bar T'$ such that
		\[
		\epsilon \tau_2^* \tau_1^*H \sim_{\mathbb{Q}} a K_{\bar{T}'} + b\tau_2^*K_{\bar{T}} + c\tau_2^*\tau_1^*K_S + E_{\bar T'},
		\]
		with $a,b\ge 0$, and $c\ge c_0>0$ with $c_0$ independent of $\epsilon$.
		Letting $\epsilon\to 0$, we obtain $K_S \equiv 0$, thus $S$ is an abelian variety by Proposition~\ref{prop:char-abel-var}.
		
		\smallskip
		Case\thinspace (1.2): The fibration $f$ is inseparable and $a_S$ is separable. By Theorem~\ref{thm:base-sep-A} we have $\epsilon H \geQ \frac{1}{2p}K_S$.  Then $K_S \equiv 0$, and thus $S$ is an abelian variety.
		
		\smallskip
		Case\thinspace (1.3): Both $f$ and $a_S$ are inseparable.
		
		We first show that $X_{K(S_1)}$ is non-reduced. Otherwise, $X_{K(S_1)}$ is integral, and by Theorem~\ref{thm:insep-A}, we have $\epsilon \tau_1^*H\geQ \tau_1^*K_{S}-K_{S_1} $.
		But this contradicts that $\kappa(S_1,\tau_1^*K_{S}-K_{S_1})\ge 1$ (Proposition~\ref{prop:basic-of-X-S-A} (2)).

		Now assume that $X_{K(S_1)}$ is non-reduced. We define $X_1$ to be the normalization of $(X_{K(S_1)})_{\rm red}$ and denote by $S'_1$ the normalization of $S_1$ in $X_1$. Then by Theorem~\ref{thm:insep-A} we have $\tau'^*_1 \tau^*_1 \epsilon H \geQ \frac12 K_{S_1'}$.
		Letting $\epsilon\to 0$, we see that $K_{S_1'} \equiv 0$ and thus $S_1'$ is an abelian variety by Proposition~\ref{prop:char-abel-var}. By the universal property of the Albanese morphism $X^{\frac1p} \to A^{\frac1p}$, we conclude that $S'_1  = S_1 \cong A^{\oneoverp}$, which means that $f_1\colon X_1 \to S_1$ is a fibration.
		
		\medskip
		Case\thinspace (2): $(K_X+\Delta)|_{X_{K(S)}}\sim_\mathbb Q 0$.
		
		By Lemma~\ref{lem:relativetrivial}, there exists a big open subset $S^{\circ} \subseteq S^{\rm reg}$ and a $\mathbb{Q}$-divisor $D^{\circ}$ on $S^{\circ}$, such that
		$-(K_X+\Delta)|_{f^{-1}(S^{\circ})}\sim_{\mathbb{Q}} f^*D^{\circ}$ and that $D:=\overline{D^\circ}$ is pseudo-effective.
		
		We first show that if at least one of $f$ and $a_S$ is separable, then $S$ is isomorphic to an abelian variety. Indeed, by applying Theorem~\ref{thm:sep-cb-formula} or Theorem~\ref{thm:base-sep-A} to the fibration $f^{\circ}\colon X_{S^{\circ}} \to S^{\circ}$, we obtain that $-D^{\circ} \geQ \frac12 K_{S^{\circ}}\geQ 0$.
		Thus both $D$ and $-D$ are pseudo-effective, so $D^\circ \simQ 0$ and therefore $K_S \simQ 0$, which implies that $S$ is isomorphic to an abelian variety.
		
		Finally, assume both $f$ and $a_S$ are inseparable.
		We can prove the statements (ii) and (iii) by applying a similar argument as in Case~(1.3).
		\end{proof}
		
		To finish the proof of Theorem~\ref{thm:S-abelian}, let us treat the remaining case. We argue by contradiction and assume $\kappa(S) >0$. We only need to consider the case
		\begin{itemize}
		\item $X_{K(S_1)}$ is non-reduced, $S\to A$ is inseparable and $S_1 = A^{\oneoverp}$.
		\end{itemize}
		
		By Proposition \ref{prop:basic-of-X-S-A} (4), we may write that
		$$\pi^*K_X  \sim_{\mathbb{Q}} K_{X_1} + (p-1)\mathfrak{C},$$
		where the movable part $\mathfrak M$ of the linear system $\mathfrak{C} = \mathfrak M + V$ has $\deg_{K(S_1)} \mathfrak M >0$, and then
		$$\pi^*(K_X + \Delta) \sim_{\mathbb{Q}}  K_{X_1} + \mathfrak M + \Delta_1,$$
		where $\Delta_1 = \pi^*\Delta + (p-2)M_0 + (p-1)V$ for some $M_0\in\mathfrak M$.
		
		\medskip{
		Claim. \it $\mathfrak{M}$ is base-point-free with $\nu(\mathfrak M) =1$ and hence induces a fibration
		$g\colon X_1\to \mathbb P^1$. Moreover, a general fiber $G_t$ of $g$ is isomorphic to an abelian variety, and $\Delta_1|_{G_t} \simQ 0$.
		}\medskip
		
		This assertion can be deduced by applying Proposition~\ref{prop:str-num0} to the pair $(X_1, \mathfrak{M} + \Delta_1)$, provided that the following condition holds:
		\begin{itemize}
		\item
		if $((X_1)_{K(S_1)}, (\Delta_1)_{K(S_1)})$ is not klt, say, $T_1$ is the unique irreducible horizontal component of $\Delta_1$, then $T_1|_{T^\nu_1}$ is pseudo-effective, where $T^\nu_1$ is the normalization of $T_1$.
		\end{itemize}
		Let $T$ be the prime divisor supported on $\pi(T_1)$ and $T^\nu$ denote the normalization of $T$. Denote by $\pi_{T^\nu_1}\colon T^\nu_1 \to T^\nu$ the natural morphism. We may write that $\pi^*T = c T_1$ for some $c>0$.
		Let $a$ be the coefficient of $T$ in $\Delta$. Then $a<1$ and by Lemma~\ref{lem:log-adj} we have
		$$cT_1|_{T^\nu_1} \sim_{\mathbb{Q}} \pi_{T_1^\nu}^*(T|_{T^\nu})
		\sim_{\mathbb{Q}} \pi_{T^\nu_1}^*\biggl(\frac{1}{1-a}\Bigl(K_{T^\nu} + B_{T^\nu} - (K_X+\Delta)|_{T^\nu}\Bigr)\biggr).$$
		This divisor is pseudo-effective since $K_{T^\nu} \geQ 0$  and $- (K_X+\Delta)$ is assumed nef.

		\smallskip
		Granted the above Claim, the condition $\Delta_1|_{G_t} \simQ 0$ implies that
		$$\det \mathcal{F}_{X_1/X}|_{G_t}= -\mathfrak{C}|_{G_t} \simQ 0.$$
		Applying Lemma~\ref{lem:fol-fibr}, we can show that $S$ is an abelian variety, which contradicts our assumption.

		\subsection{Proof of Theorem~\ref{thm:main-thm-cbf}} \label{sec:proof-cbf}
		If the Albanese morphism $a_S\colon S \to A$ is separable, we can apply Theorem~\ref{thm:base-sep-A}. If $a_S$ is inseparable, we have $\kappa(S,D)\ge0$ by Theorem~\ref{thm:insep-A}.
		
		Now assume that $a_S\colon S \to A$ is finite and inseparable and that $(X_{K(S)},\Delta_{K(S)})$ is klt. To show $\kappa(S,D)\ge1$, we argue by contradiction, assuming that $\kappa(S,D)=0$.
		Then by Theorem~\ref{thm:insep-A}, the Frobenius map $A^{\oneoverp}\to A$ factors into $ A^{\oneoverp}\buildrel\tau\over\to S \buildrel a_S\over\to A$.
		By Covering Theorem~\ref{thm:covering}, $\kappa(A^{\oneoverp}, \tau^* D) = \kappa(S,D)=0$.
		Note that any effective divisor on an abelian variety is semi-ample, thus $\tau^*D \simQ 0$.
		This implication yields $D \sim_{\mathbb{Q}} 0$, which leads to $K_X + \Delta \equiv 0$.
		Applying Theorem~\ref{thm:S-abelian}, we deduce that $S$ must be an abelian variety.
		Consequently, the morphism $a_S$ becomes an isomorphism, a contradiction.

\section*{Acknowledgments}
		The authors would like to thank the referee for giving many helpful comments to correct errors and improve the presentation and the proof. This research is partially supported by CAS Project for Young Scientists in Basic Research (No. YSBR-032), National Key R and D Program of China (No. 2020YFA0713100) and NSFC (No.12122116 and No. 12471495).
		The first author is also supported by Hubei Minzu University (Grant No. XN24040).

\bibliographystyle{plain}

	\end{document}